\documentclass[a4paper,reqno,oneside,10pt]{amsart}
\usepackage[utf8]{inputenc}
\usepackage{amssymb,amsfonts,amsmath,amsthm,eucal,url,cite}
\usepackage{enumitem}
\usepackage{tikz-cd}
\usepackage{hyperref}
\usepackage{multicol}
\usepackage{slashed}

\allowdisplaybreaks[4]

\newcommand{\subscript}[2]{$#1 _ #2$}

\newcommand{\dd}{\mathrm{d}}
\renewcommand{\Tilde}{\widetilde}

\renewcommand{\Bar}{\overline}

\newcommand{\RR}{\mathbb{R}}
\newcommand{\CC}{\mathbb{C}}
\newcommand{\QQ}{\mathbb{Q}}
\newcommand{\NN}{\mathbb{N}}
\newcommand{\ZZ}{\mathbb{Z}}

\def\R{\mathrm{R}}

\def\Span{\mathrm{Span}}

\newtheorem{theorem}{Theorem}[section]
\newtheorem{prop}[theorem]{Proposition}
\newtheorem{cor}[theorem]{Corollary}
\newtheorem{lemma}[theorem]{Lemma}
\newtheorem{definition}[theorem]{Definition}

\theoremstyle{definition}

\theoremstyle{remark}
\newtheorem{rem}[theorem]{\bf Remark}
\newtheorem{example}[theorem]{\bf Example}

\title{Locally conformally product structures}
\author[B. Flamencourt]{Brice Flamencourt}
\address[Brice Flamencourt]{Institut für Geometrie und Topologie, Fachbereich Mathematik, Universität Stuttgart, Pfaffenwaldring 57, 70569 Stuttgart, Germany
}
\email{brice.flamencourt@igt.uni-stuttgart.de}

\subjclass[2010]{53C29, 53C12, 11R18, 11R04}
\keywords{Conformal geometry, Weyl structure, OT-manifolds, number fields}

\begin{document}

\maketitle

\begin{abstract}
A locally conformally product (LCP) structure on compact manifold $M$ is a conformal structure $c$ together with a closed, non-exact and non-flat Weyl connection $D$ with reducible holonomy. Equivalently, an LCP structure on $M$ is defined by a reducible, non-flat, incomplete Riemannian metric $h_D$ on the universal cover $\Tilde M$ of $M$, with respect to which the fundamental group $\pi_1(M)$ acts by similarities. It was recently proved by Kourganoff that in this case $(\tilde M, h_D)$ is isometric to the Riemannian product of the flat space $\mathbb{R}^q$ and an incomplete irreducible Riemannian manifold $(N,g_N)$. In this paper we show that for every LCP manifold $(M,c,D)$, there exists a metric $g\in c$ such that the Lee form of $D$ with respect to $g$ vanishes on vectors tangent to the distribution on $M$ defined by the flat factor $\mathbb{R}^q$, and use this fact in order to construct new LCP structures from a given one by taking products. We also establish links between LCP manifolds and number field theory, and use them in order to construct large classes of examples, containing all previously known examples of LCP manifolds constructed by Matveev-Nikolayevsky, Kourganoff and Oeljeklaus-Toma (OT-manifolds).
\end{abstract}

\section{Introduction}

On any Riemannian manifold, there exists a unique torsion-free metric connection, called the Levi-Civita connection, which is the basic tool of Riemannian geometry. However, if one consider the slightly more general context of conformal geometry, the uniqueness of compatible connection does not hold anymore.

Conformal structures were introduced in 1919 by Weyl in the third edition of the book {\em Raum, Zeit, Materie} \cite{Weyl}, in an attempt to unify electromagnetism and gravity. He defined conformal classes of Riemannian metrics, and considered the set of torsion-free compatible connections, nowadays called Weyl structures. The fundamental theorem of conformal geometry states that they form an affine space modelled on the space of one-forms.

In general, a Weyl structure does not preserve any metric in the conformal class, even locally. Those which satisfy this property in a neighbourhood of each point are called closed, and those which preserve a global metric are called exact Weyl structures. In this article we are mostly interested in the closed, non-exact Weyl structures on compact conformal manifolds.

The study of closed Weyl structures on a conformal manifold $M$ can be better understood in terms of the universal cover $\Tilde M$. Indeed, the lift of a closed Weyl structure $D$ to $\Tilde M$ is exact, meaning that it is the Levi-Civita connection of a Riemannian metric $h_D$ on $\Tilde M$, uniquely defined up to a constant factor and consistent with the lift of the given conformal structure. The fundamental group of $M$ acts by $h_D$-similarities on $\Tilde M$, all of them being isometries if and only if $D$ is exact.

Every geometrical property of the closed Weyl connection $D$ can be interpreted on the Riemannian manifold $(\Tilde M, h_D)$, and conversely. One natural question to study is the reducibility of the holonomy group of $D$, or equivalently of the Riemannian metric $h_D$.

A first step in this direction was done by Belgun and Moroianu in \cite{BeMo}, where the authors, motivated by a result of Gallot \cite{Gal}, conjectured that a closed non-exact Weyl structure on a compact conformal manifold has reducible holonomy if and only if it is flat. They showed that the conjecture holds under an additional assumption about the lifetime of half-geodesics on the universal cover. However, soon after the formulation of the conjecture, a counter-example was  proposed by Matveev and Nikolayevsky \cite{MN} who constructed a cocompact action by a group of similarities on the Riemannian product of an Euclidean space and an incomplete irreducible Riemannian manifold. Additionally, the same authors proved that this is the only possible type of counter-example in the analytic setting \cite{MN2}.

Recently, Kourganoff extended this result to the smooth setting \cite[Theorem 1.5]{Kou}. More precisely, he proved that if a closed, non-exact Weyl structure $D$ on a compact conformal manifold $(M,c)$ is non-flat and has reducible holonomy, then the Riemannian manifold $(\Tilde M, h_D)$ is isometric to the Riemannian product $\RR^q \times (N, g_N)$ where $\RR^q$ (the {\em flat part}) is an Euclidean space and $(N, g_N)$ (the {\em non-flat part}) is an irreducible, non-complete manifold. In this case, $(M,c,D)$ is called a {\it locally conformally product} structure, or {\it LCP} structure for short. This article is devoted to the study of these particular structures on compact manifolds.

There are up to now only few examples of LCP manifolds. As mentioned before, the first one was given in \cite{MN}, and generalized in \cite[Example 1.6]{Kou} (we outline the construction in Example~\ref{exKou} below). This example is very restrictive because it only provides LCP manifolds of dimension $3$ or $4$, with a flat part of dimension $1$ or $2$ \cite{MMP}. Nevertheless, they are the only examples when the non-flat part is of dimension $2$ \cite[Theorem 1.8]{Kou}.

The other class of example comes from the theory of {\it locally conformally Kähler} (or {\it LCK}) manifolds. A conformal complex manifold is LCK if for any point there exists a metric in the conformal class which is Kähler in a neighbourhood of this point. This is equivalent to the existence of a Kähler metric on the universal cover, which belongs to the lift of the conformal class. In \cite{OT}, Oeljeklaus and Toma constructed a class of complex manifolds called OT-manifolds, some of which admit LCK structures (we recall the construction in Example~\ref{OTex} below) which turn out to be LCP structures. These LCP manifolds have flat parts of dimension $2$, so they are still restrictive examples.

One can define several invariants on LCP manifolds. On the one hand, the dimensions of the flat and the non-flat parts, and on the other hand, the rank of the subgroup of $\RR_+^*$ generated by the similarity ratios of $\pi_1 (M)$ acting on $(\Tilde M, h_D)$, which we call the {\it rank} of the LCP manifold. As noticed before, in the known examples the possibilities for these numbers are limited: the flat part is always of dimension $1$ or $2$, and it is not clear whether or not the rank can be higher than $1$. Our first goal in the present text is to extend the examples of LCP manifolds, and to show, in particular, that the three invariants previously introduced can be chosen arbitrarily large.

Let us now describe the organization of the paper and the results within. In Section~\ref{prelLCP}, we recall the background of Weyl structures and we define LCP manifolds. We also remind some basics about algebraic number fields, which will be needed in the sequel. Indeed, it turns out that the study of LCP manifolds is closely related to number theory, a fact that we can already notice from the previous examples, which involve matrices in $\mathrm{GL}_n (\ZZ)$ \cite{Kou} and algebraic number fields \cite{OT}. The structure theorem for LCP manifold proved by Kourganoff \cite[Theorem 1.9]{Kou}, is also restated. This last article will actually be our main tool, so we will often refer to it in the subsequent lines.

Section~\ref{secpropLCP} is devoted to the proof of several properties of LCP manifolds. First, we prove in Proposition~\ref{flatpartfun} that there exists a metric in the conformal class $c$ on $M$ with respect to which the Lee form of the Weyl structure $D$ vanishes on the flat distribution of $D$. This property is equivalent to the existence of a smooth function defined on the non-flat factor $N$, having the same equivariance as the metric $h_D$ on $\Tilde M$ with respect to the action of $\pi_1(M)$. In turn, the existence of such functions allows us to construct, starting from a given compact LCP manifold $(M,c,D)$, infinitely many new examples, by taking the product of $\Tilde M$ with the universal cover of a compact manifold, endowed with a warped product metric admitting a free cocompact action by similarities. This leads to the concept of reducible LCP manifolds. Moreover, in Proposition~\ref{similratio} we exhibit the link between number theory and LCP structures by proving that the similarity ratios of $\pi_1 (M)$ acting on $(\Tilde M, h_D)$ are always units in some algebraic number field. This particular fact enlighten the construction of the previous examples and their {\it a priori} surprising use of number theory, as well as it gives a direction to construct other LCP manifolds. 

In Section~\ref{SectionExLCP}, we give new examples of LCP manifolds. Extending the ideas of the OT-manifolds, we construct a class of particular manifolds on which we can define LCP structures, such that this class contains all previously introduced examples of LCP manifolds. Using some Galois theory and Dirichlet's unit theorem, we construct LCP manifolds with arbitrary rank (Proposition~\ref{rankLCP}) belonging to this class. We also find LCP manifolds with flat and non-flat part of arbitrarily large dimension, proving that the invariants previously defined can be chosen arbitrarily. In addition, the previous considerations allows us to construct an LCP structure on any OT-manifold by showing that they are a particular case of the class we just defined.

\subsection*{Acknowledgements} The author would like to thank the anonymous reviewer of this article for reading it very carefully and for the suggestions made. He also thanks the referees of his PhD thesis, Vestislav Apostolov and Michal Wrochna, for their remarks for the improvement of the article, as well as his advisor Andrei Moroianu for his help during the redaction.

\section{Preliminaries} \label{prelLCP}

\subsection{Locally conformally product manifolds} Let $M$ be a smooth manifold of dimension $n$ and denote by $\mathrm{Fr} (M)$ its frame bundle. For every $k \in \RR$ we define the weight bundle $\mathcal L^k := \mathrm{Fr} (M) \times_{\vert \mathrm{det} \vert^\frac{k}{n}} \RR$, which is an oriented bundle.

A {\em conformal class} on $M$ is a positive definite section of the fibre bundle $\mathrm{Sym} (T^* M \otimes T^* M) \otimes \mathcal L^2$. The manifold $M$ together with this section is called a conformal manifold. Equivalently, a conformal manifold is given by $M$ and a class of metrics $c$ which are related in the following manner: for any $g, g' \in c$, there is $f : M \to \RR$ such that $g' = e^{2 f} g$.

On a conformal manifold, there is no preferential connection as in the Riemannian case with the Levi-Civita connection, because the metric is defined up to multiplication by a positive function. However, a new class of connections is relevant:

\begin{definition} \label{Weylstructure}
A Weyl structure on a conformal manifold $(M,c)$ is a torsion-free connection $D$ on $TM$ which preserves $c$ i.e. such that for any $g \in c$, there is a $1$-form $\theta_g$ on $M$, called the Lee form of $D$ with respect to $g$, satisfying $D g = -2 \theta_g \otimes g$.
\end{definition}

It comes from the definition that if $\theta_g$ is the Lee form of $D$ with respect to $g \in c$, then for any $g' := e^{2f} g \in c$, the Lee form of $D$ with respect to $g'$ is $\theta_g - d f$. Then, the Lee form of $D$ with respect to $g$ is closed (resp. exact) if and only if the Lee form of any metric in $c$ is closed (resp. exact). For this reason, we introduce the following terminology:

\begin{definition}
A Weyl structure $D$ on a conformal manifold $(M,c)$ is closed (resp. exact) if the Lee form of at least one metric (and then of all metrics) in $c$ is closed (resp. exact).
\end{definition}

An easy consequence of the definition is that a closed Weyl structure is locally the Levi-Civita connection of a metric in $c$, and an exact Weyl structure is the Levi-Civita connection of a metric in $c$.

We recall that a similarity between two Riemannian manifolds $(M_1,g_1)$ and $(M_2,g_2)$ is a diffeomorphism $s : M_1 \to M_2$ such that $s^* g_2 = \lambda^2 g_1$ for some positive real number $\lambda > 0$ called the similarity ratio. In order to define the main object of this text, we need the following definition:

\begin{definition}
A similarity structure on a compact manifold $M$ is a metric $h$ on its universal cover $\Tilde M$ such that $\pi_1(M)$ acts by similarities on $(\Tilde M, h)$. A similarity structure is said to be Riemannian if in addition $\pi_1 (M)$ acts only by isometries.
\end{definition}

It turns out that this notion is closely related to closed Weyl structures. More precisely, we have the following result:

\begin{prop} \label{SimEqD}
On a conformal manifold $(M, c)$ there is a one-to-one correspondence between closed Weyl structures and similarity structures $h$ in the lifted conformal structure $\Tilde c$ on the universal cover, defined up to multiplication by a positive real number. This correspondence takes exact Weyl structures to Riemannian similarity structures.
\end{prop}
\begin{proof}
Let $D$ be a closed Weyl structure on $(M,c)$. Let $\Tilde M$ be the universal cover of $M$ and $\Tilde c$ the induced conformal structure on $\Tilde M$. The connection $D$ induces a Weyl structure $\Tilde D$ on $\Tilde M$ which is exact since $\Tilde M$ is simply connected. Thus, there is a metric $h_D \in \Tilde c$, unique up to multiplication by a positive number, such that $\nabla^{h_D} = \Tilde D$, where $\nabla^{h_D}$ is the Levi-Civita connection of $h_D$. If $g \in c$ is a metric on $M$, the induced metric $\Tilde g$ on $\Tilde M$ can be written $\Tilde g = e^{-2f} h_D$ for some real-valued function $f$ of $\Tilde M$, and a simple calculation shows that the Lee form of $\Tilde D$ with respect to $\Tilde g$ is $d f$, which means that the pull-back $\Tilde \theta_g$ of the Lee form $\theta_g$ is equal to $d f$. Now, let $\gamma \in \pi_1(M)$. One has $d f = \Tilde \theta_g = \gamma^* \Tilde \theta_g = \gamma^* d f$, thus there is $\lambda > 0$ such that $\gamma^* f = f + \ln \lambda$ and $\gamma^* h_D = \lambda^2 h_D$. We conclude that the elements of $\pi_1(M)$ act on $(\Tilde M, h_D)$ as similarities. Moreover, if these similarities are all isometries, the Weyl structure $D$ is exact.

Conversely, assume one has a compact manifold $M$ and a metric $h$ on its universal cover $\Tilde M$ such that $\pi_1(M)$ acts by similarities on $(\Tilde M, h)$. Then, the metric $h$ does not define a metric on $M$, but it induces a conformal class $c$, and the Levi-Civita connection $\nabla^h$ descends to a closed Weyl structure on $(M,c)$. If the elements of $\pi_1(M)$ are all isometries, this Weyl structure is exact.
\end{proof}

As we mentioned in the introduction, it was conjectured by Belgun an Moroianu \cite{BeMo} that given a conformal manifold together with a closed, non-exact Weyl structure, the induced connection on the universal cover must be flat or irreducible. A counter-example to this conjecture was found by Matveev and Nikolayevsky \cite{MN}, who showed that in the non-flat, analytic case, the universal cover is a Riemannian product $\RR^q \times N$ where $q \ge 0$ and $N$ is a non-complete, irreducible manifold of dimension at least $2$. This result was extended by Kourganoff to the smooth setting. More precisely, he proved the following theorem \cite[Theorem 1.5]{Kou}:
\begin{theorem} \label{classWeylstruc}
Consider a compact manifold $M$ endowed with a non-Riemannian similarity structure, and its universal cover $\Tilde M$ is equipped with the corresponding Riemannian metric $h_D$ ($D$ being the closed non-exact Weyl structure associated via Proposition~\ref{SimEqD}). Then we are in exactly one of the following situations:
\begin{enumerate}
\item $(\Tilde M, h_D)$ is flat.
\item $(\Tilde M, h_D)$ has irreducible holonomy and $\mathrm{dim} (\Tilde M) \ge 2$.
\item $(\Tilde M, h_D) = \RR^q \times (N, g_N)$, where $q \ge 1$, $\RR^q$ is the Euclidean space, and $(N, g_N)$ is a non-flat, non-complete Riemannian manifold which has irreducible holonomy.
\end{enumerate}
\end{theorem}
In the third case of Theorem~\ref{classWeylstruc}, we say that $M$ is a {\it locally conformally product manifold}, or {\it LCP manifold} for short. Then, a LCP manifold $(M,c,D)$ is the data of a compact manifold, a conformal class, and a closed, non-exact Weyl structure, with reducible, non-flat holonomy.

\begin{rem} \label{CauchyborderLCP}
We recall that the Cauchy border of a Riemannian manifold $\mathcal Z$ is $\partial \mathcal Z := C \mathcal Z \setminus \mathcal Z$, where $C \mathcal Z$ is the metric completion of $\mathcal Z$. The classification of flat similarity structures was done in \cite{Frie}. From this result, it comes that in the first case of Theorem~\ref{classWeylstruc}, the Cauchy border of $\Tilde M$ must be a single point. But this cannot happen in the case of an LCP manifold, because the flat part is a Riemannian factor of $\Tilde M$ and the non-flat part is incomplete, so $\partial \Tilde M$ must have infinite cardinal. A direct consequence of this observation is that on a compact conformal manifold $(M,c)$, a closed, non-exact Weyl structure $D$ defines an LCP structure if and only if $(\Tilde M, h_D)$ (where $\Tilde M$ is the universal cover of $M$, and $h_D$ is the similarity structure induced by $D$) has reducible holonomy and infinite Cauchy border, or equivalently if $(\Tilde M, h_D)$ has a flat Riemannian factor $\RR$.
\end{rem}

We will often write the universal cover of an LCP manifold $(M, c, D)$ as $(\Tilde M, h_D) = \RR^q \times (N, g_N)$. In this case, $\RR^q$ will always stand for the flat part of the de Rham decomposition of $\Tilde M$, $(N, g_N)$ is the non-flat, incomplete, irreducible part, and $h_D$ is the similarity structure induced by $D$, defined up to a constant factor.

We define the following invariant on LCP manifolds:
\begin{definition}
The rank of an LCP manifold $(M,c,D)$ is the rank of the subgroup of $\RR^*_+$ generated by the ratios of the elements of $\pi_1(M)$ viewed as similarities acting on $(\Tilde M, h_D)$.
\end{definition}

Equivalently, the rank of an LCP manifold $(M,c,D)$ is the minimal rank of a subgroup of $H^1(M, \ZZ)$ whose span in $H^1(M, \RR)$ contains the cohomology class $[ \theta ]$ of the Lee form of $D$. To prove this last fact, we recall that there is a canonical isomorphism
\begin{align}
\Xi : H^1 (M, \RR) \to \mathrm{Hom} (\pi_1 (M), \RR), && [w] \mapsto \left( [\gamma] \mapsto \int_\gamma w \right).
\end{align}
This map induces an isomorphism from $H^1 (M, \ZZ)$ to $\mathrm{Hom} (\pi_1 (M), \ZZ)$. In addition, an easy computation shows that $\Xi ([ \theta ])$ is exactly the composition of the logarithm and the morphism associating to an element of $\pi_1 (M)$ its similarity ratio, so the rank of the image of $\Xi ([ \theta ])$ is the rank of the LCP manifold, denoted by $r$. Since $\Xi$ is an isomorphism, it is sufficient to prove that the rank $s$ of the smallest subgroup of $\mathrm{Hom} (\pi_1(M), \ZZ)$ whose span contains $\Xi(\theta)$ is $r$. Since the image of $\Xi (\theta)$ is of rank $r$, there exist $r$ morphisms $p_1, \ldots, p_r \in \mathrm{Hom}(\pi_1(M), \RR)$, whose images are of rank $1$, such that $\Xi (\theta) = \sum_{k = 1}^r p_k$. For all $1 \le k \le r$, there is $a_k \in \RR$ such that $p_k = a_k p_k'$ where $p_k' \in \mathrm{Hom}(\pi_1(M), \ZZ)$. Consequently, $\Xi (\theta) = \sum_{k = 1}^r a_k p_k'$, thus $s \le r$. In addition, $r \le s$ because if $\Xi (\theta) = \sum_{k = 1}^s p_k$ with the $p_k$'s being morphisms with images of rank $1$, then the rank of the image of $\Xi (\theta)$ is smaller than $s$.

A first example of LCP manifold was given by Matveev and Nikolayevsky \cite{MN} and generalized by Kourganoff \cite[Example 1.6]{Kou}. We outline it here:

\begin{example} \label{exKou}
Let $\Tilde M := \RR^{q+1} \times \RR^*_+$ with $q \ge 1$. Let $b$ be a symmetric positive definite bilinear form on $\RR^{q+1}$ and $A \in \mathrm{SL}_{q+1}(\ZZ)$ such that there exist $\lambda \in (0,1)$ and a decomposition $\RR^{q+1} = E^u \perp E^s$ (where the orthogonal symbol refers to the metric induced by $b$) stable by $A$ with $A\vert_{E^s} = \lambda O$ where $O \in O(E^s, b\vert_{E^s})$, and $E^u$ is one-dimensional.

Let $G$ be the group of transformations of $\Tilde M$ generated by the translations $\RR^{q+1} \times \RR^*_+ \ni (x, t) \mapsto (x + e_k, t)$, $k \in \{1, \ldots, q+1 \}$ where $e_k$ is the $k$-th vector of the canonical basis of $\RR^{q+1}$, and the transformation $\RR^{q+1} \times \RR^*_+ \ni (x, t) \mapsto (A x, \lambda t)$.

Let $\varphi : \RR_+^* \to \RR_+^*$ be a function satisfying $\varphi (\lambda t) = \lambda^{2q+2} \varphi (t)$. We define a metric $h$ on $\Tilde M$ by
\[
h_{x,t} := b \vert_{E^s} + \varphi(t) b \vert_{E^u} + d t^2
\]
for any $(x,t) \in \Tilde M$. Then, the metric $h$ defines a similarity structure on the manifold $\Tilde M / G$.
\end{example}

However, as it was pointed out in \cite[Proposition 1]{MMP}, the only admissible values of $q$ in Example~\ref{exKou} are $q = 1,2$, so this construction only provides examples of LCP manifolds of dimension $3$ or $4$.

In the remaining part of this section, $(M, c, D)$ is an LCP manifold, and $(\Tilde M, h_D) = \RR^q \times (N, g_N)$ is its universal cover.

Let $\gamma \in \pi_1(M)$. Since $\gamma$ acts as a similarity on $(\Tilde M, h_D)$, it must preserve the de Rham decomposition, meaning that there is a similarity $\gamma_E$ (for Euclidean) of $\RR^q$ and a similarity $\gamma_N$ of $N$ such that $\gamma = (\gamma_E, \gamma_N)$.

Thus, we introduce the following definitions:
\begin{definition} \label{defP}
We define $P = \lbrace p \in \mathrm{Sim}(N), \exists \gamma \in \pi_1(M), \gamma_N = p \rbrace$, the restriction of $\pi_1(M)$ to the non-flat part $N$. We also introduce $\Bar P$, ${\Bar P}^0$ which are respectively the closure of $P$ in $\mathrm{Sim} (N)$, and the identity connected component of this closure.
\end{definition}
The groups considered in Definition~\ref{defP} were introduced by Kourganoff in \cite{Kou}, and their analysis provides several useful results on LCP manifolds. We will keep these notations throughout this text. From \cite[Lemma 4.1]{Kou} we know that $\Bar{P}^0$ is abelian and by \cite[Lemma 4.13]{Kou} that ${\Bar P}^0$ acts on $N$ by isometries.

There is actually a correspondence between $P$ and $\pi_1(M)$:

\begin{lemma} \label{isomPpi}
The group $P$ is isomorphic to $\pi_1(M)$.
\end{lemma}
\begin{proof}
The second projection $\pi_1(M) \to P, \gamma \mapsto \gamma_N$ is a group morphism. We will show that it is an isomorphism. Assume there is $\gamma \in \pi_1(M) \setminus \mathrm{id}$ such that $\gamma_N = \mathrm{id}$, so $\gamma$ is an isometry of $(\Tilde M, h_D)$ since $\gamma$ is a similarity. Let $v \in \pi_1(M)$ whose similarity ratio is $\lambda \in (0,1)$. By the Banach fixed point theorem, $v_E$ has a fixed point, and we can assume without loss of generality that it is $0$. Then, we can find $R_v, R_\gamma \in O_q (\RR)$ and $t_\gamma \in \RR^q$ such that $v_E (a)= \lambda R_v a$ and $\gamma_E (a) = R_\gamma a + t_\gamma$ for any $a \in \RR^q$. Since $\gamma$ cannot have a fixed point, because $\pi_1(M)$ acts freely on $\Tilde M$, one has $t_\gamma \in \mathrm{ker} (R_\gamma - \mathrm{Id}) \setminus \{ 0 \}$.

One has, for any $k \in \NN$ and $(a,x) \in \RR^q \times N$:
\begin{equation}
v^k \gamma v^{-k} (a,x) = (R_v^k R_\gamma R_v^{-k} a + \lambda^k R_v^k t_\gamma, x).
\end{equation}
Since $v^k \gamma v^{-k} (0,x) = (\lambda^k R_v^k t_\gamma, x) \underset{k \to + \infty}{\longrightarrow} (0,x)$, the orbit of $(0,x)$ by $\pi_1(M)$ admits an accumulation point, which contradicts the fact that $\pi_1(M)$ acts properly on $\Tilde M$.
\end{proof}

\subsection{Number theory} We will need a few notions coming from number theory in order to give examples of locally conformally product manifolds having arbitrary high rank.

First, we recall that an algebraic number field $K$, or number field for short, is an extension of $\QQ$ of finite dimension. The degree $[K : \QQ]$ of such an extension is its dimension as $\QQ$-vector space. If $\alpha$ is an algebraic number, we will denote by $\QQ[\alpha]$ the smallest extension of $\QQ$ containing $\alpha$. In this case, the degree of $\alpha$ is the degree of its (monic) minimal polynomial. The conjugates of an algebraic number $\alpha$ are the roots of its minimal polynomial.

\begin{definition}
An algebraic number field is called totally real if all its embeddings in $\CC$ lie in $\RR$.
\end{definition}

Equivalently, a number field $K := \QQ[\alpha]$ is totally real if and only if the minimal polynomial of $\alpha$ has only real roots, i.e. all the conjugates of $\alpha$ are real.

We recall that an extension $K/L$ is a Galois extension if it is normal, meaning that all the conjugates of an element $\alpha \in K$ lie in $K$, and separable, i.e. the minimal polynomial of any $\alpha \in K$ has simple roots in an algebraic closure of $K$. In this case, the Galois group of $K/L$ is the set of automorphisms of $K$ which fixes $L$. When $L = \QQ$, all the algebraic extensions are separable, so for an extension $\QQ[\alpha]$, to be a Galois extension means that all the conjugates of $\alpha$ lie in $\QQ[\alpha]$. These considerations lead us to introduce the following definition:

\begin{definition}
An extension $K/L$ is called cyclic if it is a Galois extension and its Galois group is cyclic.
\end{definition}

One object of interest for our analysis will be the ring of integers of an extension $K$, and more specifically its group of units.

\begin{definition}
An element $\beta$ of an algebraic number field $K$ is an algebraic integer if its monic minimal polynomial is in $\ZZ [X]$. The set of all algebraic integer is called the {\em ring of integers} and is denoted by $\mathcal O_K$.
\end{definition}

One basic result is that the set $\mathcal O_K$ of the algebraic integers in $K$ is a ring. In addition, it is easily seen that for any $x \in K$ there exists $p \in \ZZ$ such that $px \in \mathcal O_K$, implying the equality $\mathrm{dim}_\QQ K = \mathrm{rk} (\mathcal O_K, +)$.

\begin{definition}
The group $\mathcal O_K^\times$ of invertible algebraic integers in $K$ is called the group of units of $K$ and its elements are called units.
\end{definition}

\begin{rem}
A useful characterization of units is the following: an algebraic integer of $K$ is a unit if and only if the constant coefficient of its minimal polynomial in $\ZZ [X]$ is equal to $\pm 1$.
\end{rem}

A fundamental result on the structure of the group of unit is the Dirichlet's units theorem (for a proof, see \cite[Theorem 5.1]{Mil}):

\begin{theorem}[Dirichlet's units theorem] \label{Dirichlet}
The group of units in a number field $K$ is finitely generated with rank equal to $s + t -1$, where $s$ is the number of real embeddings of $K$, and $2t$ is the number of nonreal complex embeddings of $K$ (so $s + 2 t = [K : \QQ]$).
\end{theorem}

In particular, $\mathcal O_K^\times \simeq T \oplus \ZZ^{s + t -1}$, where $T$ is the subgroup of torsion elements in $\mathcal O_K^\times$. When $K$ is totally real, there is no torsion element different from $\pm 1$, and then $\mathcal O_K^\times \simeq \{ \pm 1 \} \oplus \ZZ^{s -1}$.

The last notion that we need concerns the bases of an algebraic number field $K$. More precisely, we are interested in the case where $K$ admits a basis which is adapted to the ring of integers.

\begin{definition}
An algebraic number field $K$ is called monogenic if there exists a power integral basis in $K$, i.e. there is an element $\alpha \in K$ such that $\mathcal O_K = \ZZ [\alpha]$.
\end{definition}

We also recall that the $n$-th cyclotomic extension is the extension of $\QQ$ generated by a primitive $n$-th root of unity. The degree of this extension is the value of the Euler's totient function at $n$.

We now have the tools to construct the so-called OT-manifolds, which were introduced by Oeljeklaus and Toma in \cite{OT}. Before giving the construction, we emphasize that throughout this article, a {\em lattice} in an Abelian Lie group $G$ will be a discrete subgroup $H$ of $G$. If the quotient $G/H$ is compact, then $H$ will be called a {\em full lattice}.

\begin{example}[OT-manifolds] \label{OTex}
Let $K$ by a number field with $s > 0$ real embeddings $\sigma_1, \ldots, \sigma_s$ and $2 t > 0$ complex embeddings $\sigma_{s+1}, \ldots, \sigma_{s+2t}$ such that $\sigma_{s+i}$ and $\sigma_{s+t+i}$ are conjugated for any $1 \le i \le t$ (such a field always exists, see \cite[Remark 1.1]{OT}). We define the geometric representation of $K$
\[
\sigma: K \to \CC^{s+t}, a \mapsto (\sigma_1 (a), \ldots, \sigma_{s+t} (a)).
\]
The image of the ring of integers $\mathcal O_K$ of $K$ by $\sigma$ is a lattice of rank $s + 2t$ in $\CC^{s + t}$. Moreover, we consider
\[
\mathcal O_K^{\times,+} := \{a \in \mathcal O_K^\times, \sigma_i (a)  > 0, 1 \le i \le s \},
\]
and we define an action of this set on $\CC^{s+t}$ by $az := (\sigma_1(a) z_1, \ldots, \sigma_{s+t}(a) z_{s+t})$ for any $a \in \mathcal O_K^{\times,+}$. Let $U$ be a subgroup of $\mathcal O_K^{\times,+}$ such that the image of $U$ by the composition $p_{\RR^s} \circ \ell$ of the logarithmic representation
\begin{equation}
\begin{aligned}
&\ell : \mathcal O_K^{\times,+} \to \RR^{s+t}, \\
&\ell (u) := (\ln \vert\sigma_1(u) \vert, \ldots, \ln \vert \sigma_s(u) \vert, 2 \ln \vert\sigma_{s+1}(u) \vert, \ldots, 2 \ln \vert\sigma_{s+t}(u) \vert)
\end{aligned}
\end{equation}
and the projection $p_{\RR^s} : \RR^{s+t} \to \RR^s$ on the first $s$ coordinates is a full lattice. We remark that $\mathrm{ker} (p_{\RR^s} \circ \ell) \subset \{\pm 1\}$, thus $U$ has rank $s$.

Let $H := \{ z \in \CC, \mathrm{Im} (z) > 0 \}$. Combining the additive action of $\mathcal O_K$ and the multiplicative action of $U$, the group $U \ltimes \mathcal O_K$ acts freely, cocompactly and properly on $H^s \times \CC^t$. Thus, the quotient $X(K,U) := (H^s \times \CC^t) / (U \ltimes \mathcal O_K)$ is a compact manifold.

When $t=1$, the manifold $X(K,U)$ admits an LCK structure, which is determined by a Kähler potential
\[
F(z) := \prod\limits_{k=1}^s \frac{i}{z_k - \bar z_k} + \vert z_{s+1} \vert^2
\]
on its universal cover \cite{OT}. This induces in turn a similarity structure on $X(K,U)$. If this structure was Riemannian, the Kähler metric $\frac{i}{2} \partial \bar \partial F$ would descend to $X(K,U)$. This is impossible because an OT-manifold admits no Kähler metric \cite[Proposition 2.5]{OT}. In addition, from the form of the Kähler potential, the second factor $\CC^t (= \CC)$ of the  universal cover of $X(K,U)$ is a Riemannian factor. Thus, by Remark~\ref{CauchyborderLCP}, $X(K,U)$ admits an LCP structure when $t = 1$.
\end{example}

\begin{rem}
In Example~\ref{OTex}, when $s = t = 1$, the Kähler potential of the lift of the LCK metric to the universal cover is \cite{OT}
\begin{equation}
F : H \times \CC \to \RR, \quad F(z) := \frac{i}{z_1 - \bar z_1} + \vert z_2 \vert^2.
\end{equation}
Writing the Kähler form as $\frac{i}{2} \sum\limits_{k \neq l} \omega_{k l} \dd z_k \wedge \dd \bar z_l$, one has:
\begin{align*}
\omega_{1 1} = \partial_{z_1} \bar \partial_{z_1} \left( \frac{i}{z_1 - \bar z_1} + \vert z_2 \vert^2 \right) = \frac{1}{4} \left( \frac{\partial}{\partial x_1} - i \frac{\partial}{\partial y_1} \right) \left( \frac{\partial}{\partial x_1} + i \frac{\partial}{\partial y_1} \right) \frac{1}{2 y_1} = \frac{1}{4} \frac{1}{y_1^3}
\end{align*}
\[
\omega_{2 2} = 1 \qquad \omega_{1 2} = 0.
\]
Then, the metric can be rewritten as $g := \frac{1}{4 y_1^3} (\dd x_1^2 + \dd y_1^2) + (\dd x_2^2 + \dd y_2^2)$. We make the change of variable $v_1 := x_1 / 2$, $w_1 := \frac{1}{\sqrt{y_1}}$ and the metric becomes
\begin{equation}
g = (w_1^6 \dd v_1^2 + \dd w_1^2) + (\dd x_2^2 + \dd y_2^2).
\end{equation}
Moreover, the group $U$ is generated by a single unit $u \in \mathcal O^{\times, +}_K$ which satisfies $\sigma_1 (u) = \vert \sigma_2 (u) \vert^{-2}$. After the change of variable, the multiplicative action of $u$ is given, for any $(v_1, w_1, x_2 + i y_2) \in \RR \times \RR_+^* \times \CC$, by
\begin{align*}
u \cdot (v_1, w_1, x_2 + i y_2) &= (\sigma_1 (u) v_1, \sigma_1 (u)^{-\frac{1}{2}} w_1, \sigma_2 (u) (x_2 + i y_2)) \\
&= (\sigma_1 (u) v_1, \vert \sigma_2 (u) \vert w_1, \sigma_2 (u) (x_2 + i y_2)).
\end{align*}
If we look at the restriction of this action to $\RR \times \CC$ by dropping the variable $w_1$, we remark that the matrix of the transformation in a basis of the lattice $\sigma (\mathcal O_K)$ belongs to $\mathrm{SL}_3 (\ZZ)$ (see the proof of Corollary~\ref{LCPonOT} below for more details). Then, we recognize the example~\ref{exKou} in the case $q = 2$.
\end{rem}

\subsection{Foliations and LCP manifolds} \label{foliation}
A foliation of dimension $p$ of an $n$-dimensional manifold $M$ is a maximal atlas $(U_i, \phi_i)_{i \in I}$ on $M$ such that for each $i, j \in I$ the transition map $\Phi_{i,j} := \phi_j \circ \phi_i^{-1}: \phi_i(U_i \cap U_j) \to \phi_j (U_i \cap U_j)$ satisfies
\begin{equation}
\frac{\partial \Phi_{i,j}^l}{\partial x_k} = 0 \quad \textrm{for all $p + 1 \le l \le n$, and $1 \le k \le p$}
\end{equation}
where $x_k$ is the $k$-th coordinate of $\RR^n$.

A foliation induces a $p$-dimensional integrable distribution on $M$, taking at each point $x \in U_i$ the subspace of $T_x M$ given by $d \phi_i^{-1} (\phi_i(x)) (\RR^p \times\{ 0 \})$. From this, one can define the leaves of the foliation as follows: if $x \in M$, the leaf passing through $x$ is the set of all the points that can be reached from $x$ by continuous, piecewise differentiable paths whose tangent vector at each smooth point is in the distribution previously defined. For more details, see \cite{Mol}.

When the manifold $M$ is compact, one can extract a finite covering $(U_i)_{i \in J}$, $J \subset I$ such that for any $i \in J$ the open set $U_i$ is diffeomorphic to a product $V_i \times T_i$ where $V_i$ and $T_i$ are open cubes of $\RR^p$ and $\RR^{n-p}$ respectively. This induces maps $f_i : U_i \to T_i$ in a natural way, and we define the transition maps $\gamma_{ij} : f_i(U_i \cap U_j) \to f_j(U_i \cap U_j)$ by $f_j = \gamma_{ij} \circ f_i$. The disjoint union $T := \bigsqcup\limits_{i \in J} T_i$ is called the {\it transversal} of the foliation. The foliation is said to be Riemannian if there exists a metric on the transversal such that the transition maps are isometries.

In \cite[Theorem 1.9]{Kou}, it was shown that an LCP manifold carries a Riemannian foliation. More precisely, one has the following theorem:

\begin{theorem} \label{Kou1.9}
Let $(M, c, D)$ be a LCP manifold, and let $(\Tilde M, h_D) = \RR^q \times (N, g_N)$ be its universal cover endowed with the metric $h_D$ induced by $D$. Here, $(N, g_N)$ is the non-flat, irreducible factor of the de Rham decomposition. The foliation $\Tilde {\mathcal F}$ tangent to $\RR^q$ induces by projection a foliation $\mathcal F$ on $M$. Then $\mathcal F$ is a Riemannian foliation on $M$, and the closures of the leaves form a singular Riemannian foliation $\Bar {\mathcal F}$ on $M$, such that each leaf of $\Bar {\mathcal F}$ is a smooth manifold of dimension $d$, depending of the leaf, with $q < d < q+n$, where $n = dim(N)$.

Moreover, on each leaf of $\Bar {\mathcal F}$, there is a flat Riemannian metric which is compatible with the similarity structure of $M$.
\end{theorem}

\begin{definition} \label{flatpart}
In Theorem~\ref{Kou1.9}, we call the distribution tangent to the leaves of $\mathcal F$ the {\em flat distribution} on $M$, and the orthogonal distribution is called the {\em non-flat distribution}.
\end{definition}

Again, we recall several results and observations from \cite{Kou}. In the setting of Theorem~\ref{Kou1.9}, we can describe the leaves of $\mathcal F$ using the canonical surjection $\pi : \Tilde M \to M$, and the group $P$ previously defined. The leaf of $\mathcal F$ passing through $\pi (a, x)$ for $(a, x) \in \RR^q \times N$ is equal to $\pi (\RR^q \times P x)$, and its closure is $\Bar {\mathcal F}_x := \pi (\RR^q \times {\Bar P}^0 x)$  \cite[Lemma 4.11]{Kou}. By Theorem~\ref{Kou1.9}, the metric $h_D$ restricted to $\RR^q \times {\Bar P}^0 x$ descends to a metric $g_x$ on $\Bar {\mathcal F}_x$. Thus, the metric $h_D$ induces a Riemannian metric, up to a multiplicative factor, on the closure of the leaves of $\mathcal F$.

Since $\Bar{P}^0$ is abelian and acts by isometries, for any $x \in N$, the closed leaf $\Bar {\mathcal F}_x$ is the product of an Euclidean space and a flat torus. In particular, it is a complete space, which implies that an element of $\pi_1(M)$ with similarity ratio $\neq 1$ acts freely on $N / \Bar {P}^0$.

We consider the subgroup of $\pi_1(M)$ defined by $\Gamma_0 := \pi_1 (M) \cap (\mathrm{Sim} (\RR^q) \times \Bar{P}^0)$. From \cite[Lemma 4.18]{Kou}, we know that this group is a full lattice in $\RR^q \times \Bar{P}^0$ where $\RR^q$ is identified with its translations. In Example~\ref{exKou} for instance, $\Gamma_0$ is the group of translations $\ZZ^{q+1}$ acting on $\RR^{q+1}$. This observation explains why we will always consider such lattices in order to construct examples.

\section{Properties of LCP manifolds} \label{secpropLCP}

Let $(M,c,D)$ be an LCP manifold and $(\Tilde M, h_D) = \RR^q \times (N, g_N)$ be its universal cover, endowed with the similarity structure $h_D$ induced by $D$. We denote by $\pi : \Tilde M \to M$ the canonical surjection.

\subsection{Adapted metrics} In this subsection, we prove that there exists a metric $g \in c$ such that the Lee form $\theta_g$ of $D$ with respect to $g$ vanishes on the flat distribution (Definition~\ref{flatpart}) of $D$ on $M$. This is equivalent to the existence of a function of $N$ having the same equivariance (the term {\em same automorphy} is also often used in the litterature) as $h_D$ with respect to $\pi_1 (M)$. For this reason, we introduce the following definition:

\begin{definition}
Let $G$ be a group acting on a Riemannian manifold $(\mathcal Z, g_\mathcal Z)$ by similarities. A smooth function $f : \mathcal Z \to \RR$ is said to be $G$-equivariant if for every $\gamma \in G$, one has $\gamma^* e^{2f} = \lambda_\gamma^2 e^{2f}$ where $\lambda_\gamma$ is the similarity ratio of $\gamma$. Equivalently, a function $f$ is $G$-equivariant if $G$ consists of isometries of $e^{-2 f} g_\mathcal Z$.
\end{definition}

We now give an important property of the equivariant functions on the universal cover $\RR^q \times N$ of LCP manifolds: they are bounded on sets of the form $\RR^q \times K$ where $K$ is a compact subset of $N$. In order to prove this result, we recall that the Cauchy boundary $\partial \mathcal Z$ of a Riemannian manifold $\mathcal Z$ is the set $C \mathcal Z \setminus \mathcal Z$ where $C \mathcal Z$ is the metric completion of $\mathcal Z$. The Riemannian distance $d^\mathcal Z$ on $\mathcal Z$, is extended to $C \mathcal Z$ in the following natural way: if $(x_n), (y_n)$ are representatives of elements $x,y \in C \mathcal Z$ (which consists of equivalence classes of Cauchy sequences in $\mathcal Z$), $(d^\mathcal Z (x_n, y_n))_{n \in \NN}$ is a Cauchy sequence, and $d^\mathcal Z (x,y)$ is defined as the limit of this sequence. We first state the following easy lemma:

\begin{lemma} \label{Kborder}
Let $A$ be a subset of a Riemannian manifold $\mathcal Z$. Assume that $\partial Z$ is non-empty. We define $\alpha := \underset{x \in A}{\inf} d^\mathcal Z (x, \partial \mathcal Z)$ and $\beta := \underset{x \in A}{\sup} \; d^\mathcal Z (x, \partial \mathcal Z)$. Then, if $\gamma$ is a similarity of $\mathcal Z$ of ratio $\lambda \in \RR_+^*$, it extends uniquely to $C \mathcal Z$ as a uniformly continuous function on a dense subset of $C \mathcal Z$ and one has the property
\[
\forall x \in A, d^\mathcal Z (\gamma x, \partial \mathcal Z) \in [\lambda \alpha, \lambda \beta].
\]
\end{lemma}
\begin{proof}
Let $x \in A$ and $\gamma$ a similarity of $\mathcal Z$ of ratio $\lambda \in \RR_+^*$. One has, $\alpha \le d^\mathcal Z (x, \partial \mathcal Z) \le \beta$. It is easy to see from the definition that $\gamma (\partial \mathcal Z) = \partial \mathcal Z$, thus $\lambda \alpha \le d^\mathcal Z (\gamma x, \partial \mathcal Z) \le \lambda \beta$.
\end{proof}

\begin{cor} \label{cor1}
In the setting of Lemma~\ref{Kborder}, for any compact subsets $K_1,K_2 \subset \mathcal Z$, the similarity ratios of the elements of $\Gamma = \lbrace \gamma \in \mathrm{Sim} (\mathcal Z), (\gamma K_1) \cap K_2 \neq \emptyset \rbrace$ are included in a compact subset of $\RR_+^*$.
\end{cor}
\begin{proof}
Let $\rho : \mathrm{Sim} (\mathcal Z) \to \RR_+^*$ be the group morphism which associates to an element of $\mathrm{Sim} (\mathcal Z)$ its similarity ratio. We also introduce
\[
\alpha_1 := \underset{x \in K_1}{\inf} d^\mathcal Z (x, \partial \mathcal Z) \qquad \beta_1 := \underset{x \in K_1}{\sup} \; d^\mathcal Z (x, \partial \mathcal Z)
\]
and
\[
\alpha_2 := \underset{x \in K_2}{\inf} d^\mathcal Z (x, \partial \mathcal Z) \qquad \beta_2 := \underset{x \in K_2}{\sup} \; d^\mathcal Z (x, \partial \mathcal Z).
\]
Let $\gamma \in \Gamma$. By definition, there exists $x \in K_1$ such that $\gamma x \in K_2$ so in particular we have
\[
\alpha_2 \le d^\mathcal Z (\gamma x, \partial \mathcal Z) \le \beta_2.
\]
Moreover by Lemma~\ref{Kborder} one has
\[
\rho(\gamma) \alpha_1 \le d^\mathcal Z (\gamma x, \partial \mathcal Z) \le \rho(\gamma) \beta_1,
\]
which implies
\begin{align*}
\rho(\gamma) \alpha_1 \le \beta_2  && \alpha_2 \le \rho(\gamma) \beta_1,
\end{align*}
so we conclude
\[
\alpha_2 / \beta_1 \le \rho(\gamma) \le \beta_2 / \alpha_1.
\]
Thus, $\rho(\Gamma)$ is included in the compact set $[\alpha_2 / \beta_1, \beta_2 / \alpha_1]$.
\end{proof}

We have now all the tools to prove the boundedness property for equivariant functions:

\begin{lemma} \label{fbound}
Let $f : \Tilde M \to \RR$ be a smooth $\pi_1(M)$-equivariant function. Then, for any compact subset $K$ of $N$, $f$ is bounded on $\RR^q \times K$.
\end{lemma}
\begin{proof}
Let $K \subset N$ be a compact set. Since $\pi_1(M)$ acts cocompactly on $\Tilde M$, there is a compact set $C \subset \Tilde M$ such that $\pi_1(M) C = \Tilde M$. Moreover $C$ can be assumed to be equal to $C_E \times C_N$ where $C_E$ is a compact of $\RR^q$ and $C_N$ is a compact of $N$. Let
\[
\Gamma := \lbrace \gamma \in \pi_1(M), (\gamma C) \cap (\RR^q \times K) \neq \emptyset \rbrace = \lbrace \gamma \in \pi_1(M), (\gamma_N C_N) \cap K \neq \emptyset \rbrace.
\]
Let $\rho : \pi_1(M) \to \RR_+^*$ be the group morphism which associates to an element of $\pi_1(M)$ its similarity ratio. Then we can apply Corollary~\ref{cor1} with the setting $\mathcal Z := N$, $\Gamma := \Gamma_N$ (the subgroup of $\Gamma$ consisting of the projection of $\Gamma$ on the factor $N$), and we obtain that $\rho (\Gamma)$ is included in a compact set $[\alpha, \beta]$, with $\alpha, \beta > 0$. We know that $f$ is bounded on $C$, meaning there are $\alpha', \beta' \in \RR$ such that $\alpha' \le f \le \beta'$ on $C$. In addition, for any $x \in \RR^q \times K$, there is $\gamma \in \Gamma$ and $y \in C$ such that $\gamma y = x$. Thus, the equivariance property of $e^{2f}$ yields $\alpha' + \ln \alpha \le f(x) \le \beta' + \ln \beta$, which gives the desired result.
\end{proof}

For $x \in N$, let $S_x := \lbrace \gamma \in \pi_1(M) \vert \gamma_N \cdot x \in {\Bar P}^0 x \rbrace$ (${\Bar P}^0$ was defined in Definition~\ref{defP}). We recall that in Section~\ref{foliation} we defined the closed leaf $\Bar{\mathcal F}_x \subset M$, and showed that the metric $h_D$ descends to a metric $g_x$ on it. We give here a short proof of a result partially stated in the proof of \cite[Lemma 4.18]{Kou}.

\begin{lemma} \label{leavesisom}
Let $ x \in N$. Then, $\Bar{\mathcal F}_x$ is isomorphic to $(\RR^q \times \Bar{P}^0 x) / S_x$ and $S_x$ acts on $(\Tilde M, h_D)$ by isometries. Moreover, if $\gamma \in \pi_1(M)$ with similarity ratio $\lambda > 0$, there is a similarity $\Bar \gamma : (\Bar{\mathcal F}_x, g_x) \to (\Bar{\mathcal F}_{\gamma x}, g_{\gamma x})$  of ratio $\lambda$ for which the following diagram is commutative:
\begin{equation}
\begin{tikzcd}
\RR^q \times \Bar{P}^0 x \arrow[r, "\gamma"] \arrow[d, "\pi"] & \RR^q \times \Bar{P}^0 \gamma_N x \arrow[d, "\pi"] \\
\Bar{\mathcal F}_x \arrow[r, "\Bar \gamma"] & \Bar{\mathcal F}_{\gamma_N x}
\end{tikzcd}
\end{equation}
\end{lemma}
\begin{proof}
The proof of \cite[Proposition 4.16]{Kou} shows that the elements of $\pi_1(M)$ with ratio different from $1$ act freely on $N / {\Bar P}^0$. Since the set $S_x$ stabilizes $\Bar{P}^0 x$ in $N / {\Bar P}^0$, it contains only isometries.

Let $(a,y)$ and $(a',y')$ in $\RR^q \times \Bar{P}^0 x$. Assume there is $\gamma \in S_x$ such that $\gamma (a,y) = (a',y')$. By definition, $\pi (a,x) = \pi (a',y')$, thus, the application $\pi$ induces a surjective map $\phi : (\RR^q \times \Bar{P}^0 x) / S_x \to \Bar{\mathcal F}_x$. We will show that $\phi$ is injective. Assume $\phi (S_x (a,y)) = \phi (S_x (a',y'))$. Thus, one has $\pi (a,y) = \pi (a',y')$, meaning there is $\gamma \in \pi_1(M)$ such that $\gamma (a,y) = (a',y')$, implying $\gamma_N y = y'$. By definition, there are $p, p' \in \Bar{P}^0$ such that $y = p \cdot x$ and $y' = p' \cdot x$, so we obtain $\gamma_N p \cdot x = p' \cdot x$, whence $\gamma_N p \gamma_N^{-1} \gamma_N \cdot x \in \Bar{P}^0 x$. Using that $\Bar{P}^0$ is normal in $\Bar{P}$, because it is the connected component of the identity, one gets $\gamma_N \cdot x \in \gamma_N p^{-1} \gamma_N^{-1} \Bar{P}^0 x = \Bar{P}^0 x$. We conclude that $\gamma \in S_x$ and $S_x (a,y) = S_x (a',y')$, providing that $\phi$ is injective.

Now, let $\gamma \in \pi_1(M)$ with similarity ratio $\lambda$. One has
\[
\gamma (\RR^q \times \Bar{P}^0 x) = \RR^q \times \gamma_N \Bar{P}^0 x = \RR^q \times \gamma_N \Bar{P}^0 \gamma_N^{-1} \gamma_N x = \RR^q \times \Bar{P}^0 \gamma_N x,
\]
justifying the first line of the diagram. On the other hand, if $(a,y), (a',y')$ are elements of $\RR^q \times \Bar{P}^0 x$ such that there is $\gamma' \in S_x$ with $\gamma' (a,y) = (a',y')$, one has
\[
\pi \circ \gamma (a,y) = \pi (a,y) = \pi (\gamma' (a,y)) = \pi (a',y').
\]
Thus, $\pi \circ \gamma$ induces a surjective map from $(\RR^q \times \Bar{P}^0 x) / S_x$ to $(\RR^q \times \Bar{P}^0 \gamma_N x) / S_{\gamma_N x}$.
To prove that $\gamma$ descends to an isomorphism $\Bar \gamma$, it is then sufficient to prove that this map is injective, or equivalently that $S_x (a,y) = S_x (a',y')$ implies $S_{\gamma_N x} \gamma (a,y) = S_{\gamma_N x} \gamma (a',y')$. It is sufficient to show that $\gamma^{-1} S_{\gamma x} \gamma = S_x$, which follows from
\begin{align*}
\gamma^{-1} S_{\gamma_N x} \gamma &= \lbrace \gamma^{-1} \gamma' \gamma, \gamma' \in S_{\gamma_N x} \rbrace \\
&= \lbrace \gamma^{-1} \gamma' \gamma, \gamma'_N \gamma_N \cdot x \in \Bar{P}^0 \gamma_N  x \rbrace \\
&= \lbrace \gamma^{-1} \gamma' \gamma, \gamma_N^{-1} \gamma'_N \gamma_N \cdot x \in \Bar{P}^0 x \rbrace \\
&\subset S_x,
\end{align*}
using again that $\Bar{P}^0$ is a normal subgroup of $\Bar P$. The same proof shows that $\gamma S_x \gamma^{-1} \subset S_{\gamma_N x}$, so we conclude that $S_x = \gamma^{-1} S_{\gamma_N x} \gamma$, which shows the existence of $\Bar \gamma$.

We easily see that $\Bar \gamma$ is a similarity of ratio $\lambda$ using the commutative diagram and the fact that $h_D$ descends to the closure of the leaves.
\end{proof}

\begin{prop} \label{flatpartfun}
Let $(M, c, D)$ be an LCP manifold and $(\Tilde M, h_D) = \RR^q \times (N, g_N)$ be its universal cover, endowed with the similarity structure $h_D$ induced by $D$. Then, there exists a smooth $P$-equivariant function $\varphi : N \to \RR$ ($P$ was defined in Definition~\ref{defP}). In particular, if we denote by $\pi_N : \Tilde M \to N$ the second projection, $\pi_N^* \varphi$ is a $\pi_1(M)$-equivariant function on $\Tilde M$ depends only on the non-flat factor $N$.
\end{prop}
\begin{proof}
We first prove that there always exists a $\pi_1(M)$-equivariant function on $\Tilde M$. Let $g$ be any Riemannian metric on $M$ in the conformal class $c$. The pull-back $\Tilde g$ of $g$ to $\Tilde M$ satisfies $e^{2f} \Tilde g = h_D$ for a function $f : \Tilde M \to \RR$, which is clearly $\pi_1(M)$-equivariant.

By Lemma~\ref{leavesisom}, $\Bar {\mathcal F}_x \simeq (\RR^q \times {\Bar P}^0 x) / S_x$ and $S_x$ acts by isometries, so the function $f\vert_{\RR^q \times {\Bar P}^0 x}$ descends to a function ${\bar f}_x$ on $\Bar {\mathcal F}_x$. The manifold $\Bar {\mathcal F}_x$ being compact, we can define
\begin{equation}
e^{2 w(x)} := \left( \int_{\Bar {\mathcal F}_x} \dd \mu_x \right)^{-1} \left( \int_{\Bar {\mathcal F}_x} e^{2{\bar f}_x} \dd \mu_x \right),
\end{equation}
where $\dd \mu_x$ is the measure induced by the metric $g_x$. Doing this for any $x \in N$ gives a function $w : N \to \RR$. We claim that this function is bounded on any compact subset of $N$. Indeed, if $K \subset N$ is compact, by Lemma~\ref{fbound} there is a constant $\beta_K > 0$ such that $f(a,x) \le \beta_K$ for any $(a,x) \in \RR^q \times K$. Since ${\Bar P}^0$ acts by isometries,  $f(a,x) \le \beta_K$ for any $(a,x) \in \RR^q \times (P \cap {\Bar P}^0) K$, and by density this still holds for $(a,x) \in \RR^q \times {\Bar P}^0 K$. Thus, for any $x \in K$ one has $\bar f_x \le \beta_K$ and consequently $w (x) \le \beta_K$.

We now check that the function $w$ still has the desired equivariance. Let $p \in P$, and let $\lambda > 0$ be its similarity ratio with respect to the metric $g_N$. By Lemma~\ref{isomPpi}, there is a unique $\gamma \in \pi_1(M)$ such that $p = \gamma_N$. Denoting $y := p \cdot x$, Lemma~\ref{leavesisom} allows us to define $\Bar \gamma : \Bar{\mathcal  F}_x \to \Bar{\mathcal  F}_{y}$, which is a similarity of ratio $\lambda$. Thus, one has
\begin{align*}
e^{2 w(y)} &= \left( \int_{\Bar {\mathcal F}_y} \dd \mu_y \right)^{-1} \left( \int_{\Bar {\mathcal F}_y} e^{2{\bar f}_y} \dd \mu_y \right) \\
&= \left( \int_{\Bar {\mathcal F}_x} \Bar{\gamma}^*(\dd \mu_y) \right)^{-1} \left( \int_{\Bar {\mathcal F}_x} \Bar{\gamma}^*(e^{2{\bar f}_y}) \Bar{\gamma}^*(\dd \mu_y) \right) \\
&= \left( \int_{\Bar {\mathcal F}_x} \lambda^n \dd \mu_x) \right)^{-1} \left( \int_{\Bar {\mathcal F}_x} \lambda^2 e^{2{\bar f}_x} \lambda^n \dd \mu_x \right) \\
&= \lambda^2 e^{2w(x)}.
\end{align*}

However, the function $w$ is not necessarily smooth. We will use a convolution process to obtain the desired smooth equivariant function. Since the foliation $\mathcal F$ is Riemannian, one can define a complete Riemannian metric $\Tilde g_N$ on $N$ with respect to which $P$ acts by isometries (see \cite[Lemma 4.9]{Kou} for further details).

As $P$ acts cocompactly by isometries on $(N, \Tilde g_N)$, the injectivity radius $r_0$ of $(N, \Tilde g_N)$ is positive i.e. for any $x \in N$ the Riemannian exponential $\exp_x$ defined by $\Tilde g_N$ is a diffeomorphism on $B_x (r_0)$, the open ball of radius $r_0$ and center $0$ in $T_x N$. Let $0 < 3 r < r_0$ and let $\chi : \RR_+ \to \RR_+$ be a smooth plateau function in a neighbourhood of $0$, compactly supported in $[0, r]$. For every $x \in N$, let $\dd V_x$ be the measure induced on $T_x N$ by the metric $\Tilde g_N$. Consider the function $\varphi : N \to \RR$ given by
\begin{equation}
e^{2 \varphi(x)} := \int_{T_x N} e^{2 w} \circ \exp_x (v) \chi (\Vert v \Vert) \dd V_x(v),
\end{equation}
which is well-defined because the function $e^{2w}$ is bounded on any compact subset of $N$.

We claim that the function $\varphi$ is smooth. To prove this fact, we first remark that for any $x$, if one denotes by $B_N(y, a) := \exp_y (B_y(a))$ the ball of radius $0 < a < r_0$ and center $y$ in $(N, \Tilde g_N)$, for any $y \in B_N(x, r)$ the Riemannian exponential is a diffeomorphism from $B_y(2 r)$ to $B_N (y, 2 r)$ because $r < r_0/3$. In particular, $B_N(x,2 r)$ does not meet the cut-locus of $y$, and the square of the distance function $d^{\Tilde g_N}$ induced by $\Tilde g_N$ is smooth on $B_N(x,2 r)$. Consequently, we can apply a differentiation under integral argument if we remark that for $y := \exp_x (v_0) \in B_N(x, r)$ (with $v_0 \in T_x N$), one has
\begin{align*}
e^{2 \varphi(y)} &= \int_{T_y N} e^{2 w} \circ \exp_y (v) \chi (\Vert v \Vert) \dd V_y (v) \\
&= \int_{T_y N} e^{2 w} \circ \exp_x \circ \exp_x^{-1} \circ \exp_y (v) \chi (d^{\Tilde g_N} (\exp_x \circ \exp_x^{-1} \circ \exp_y (v), 0)) \dd V_y (v) \\
&= \int_{T_x N} e^{2 w} \circ \exp_x (v) \chi(d^{\Tilde g_N} (\exp_x (v), \exp_x (v_0))) (\exp_y^{-1} \circ \exp_x)^* (\dd V_y) (v) \\
&= \int_{T_x N} e^{2 w} \circ \exp_x (v) \chi(d^{\Tilde g_N} (\exp_x (v), y)) vol(y, v) \dd V_x (v)
\end{align*}
where $vol$ is a smooth function giving the change of volume element. 

It remains to check the equivariance property. Let $p \in P$, and let $\lambda > 0$ be its similarity ratio for the metric $h_D$. One has, denoting $y := p \cdot x$, and using the fact that $p$ is an isometry of $(N, \Tilde g_N)$:
\begin{align*}
e^{2 \varphi(y)} &= \int_{T_y N} e^{2 w} \circ \exp_y (v) \chi(\Vert v \Vert) \dd V_y (v) \\
&= \int_{T_x N} (p^* e^{2 w}) \circ \exp_x (v) \chi(\Vert v \Vert) p^*(\dd V_y) (v) \\
&= \int_{T_x N} \lambda^2 e^{2 w} \circ \exp_x (v) \chi(\Vert v \Vert) \dd V_x (v) \\
&= \lambda^2 e^{2 \varphi}.
\end{align*}
Then, $\varphi$ is a $P$-equivariant function.
\end{proof}

\begin{rem}
It is easy to show that the $P$-equivariant function $\varphi$ given by Proposition~\ref{flatpartfun} is in fact $\Bar P$-equivariant. Indeed, for any $p \in P \cap \Bar{P}^0$ one has $p^* \varphi =  \varphi$ since $\Bar P^0$ acts by isometries. As $P \cap \Bar P^0$ is dense in $\Bar P^0$, we actually have $\varphi = p^* \varphi$ for all $p \in {\Bar P}^0$. Our claim thus follows from \cite[Lemma 4.10]{Kou}, which states that $\Bar P = P \Bar P^0$.
\end{rem}

We define a particular class of metric on $M$:
\begin{definition} \label{adaptedDef}
A metric $g$ on $M$ with lift $\Tilde g$ on $\Tilde M$ is said to be adapted if there exists a smooth function $f : N \to \RR$ such that $e^{2f} \Tilde g = h_D$.
\end{definition}
With this definition, Proposition~\ref{flatpartfun} just states that there exist adapted metrics.

As a direct application of Proposition~\ref{flatpartfun}, we show that given a compact manifold $K$ with universal cover $\Tilde K$, it is possible to construct an LCP manifold with universal cover $\Tilde M \times \Tilde K$. Indeed, let $\varphi : N \to \RR$ be the smooth equivariant function given by Proposition~\ref{flatpartfun}. Let $g_K$ be a metric on $K$ and $\Tilde g_K$ its pull-back to $\Tilde K$. The metric
\begin{equation} \label{defhMK}
h_{M,K} := h_D + e^{2 \varphi} \Tilde g_K
\end{equation}
on $\Tilde M \times \Tilde K$ defines a similarity structure on $M \times K$, and thus an LCP structure $(M \times K, c_K, D_K)$, which proves our claim.

We give a name to the previous construction
\begin{definition} \label{extensionDef}
The LCP structure $(M \times K, c_K, D_K)$ is called an {\em extension} of $(M, c, D)$ (by $K$).
\end{definition}

\begin{prop}
Let $(M \times K, c_K, D_K)$ be an extension by $K$ of $(M, c, D)$. Then, the non-flat part of $(\Tilde M \times \Tilde K, h_{M, K})$ ($h_{M,K}$ is defined in Equation~\eqref{defhMK}) is $N \times \Tilde K$.
\end{prop}
\begin{proof}
It is easy to see that the non-flat distribution (Definition~\ref{flatpart}) of $(\Tilde M \times \Tilde K, h_{M, K})$ is a subdistribution of $T (N \times \Tilde K)$ since it has to be orthogonal to the flat distribution, and then orthogonal to $\RR^q$. From the definition of LCP manifold (see Theorem~\ref{classWeylstruc} and the definition below), $(N \times \Tilde K, g_N + e^{2 \varphi} \Tilde g_K)$ has a de Rham decomposition of the form $\RR^{q'} \times (N', g_{N'})$, where $q'$ migth be $0$ and $(N', g_{N'})$ is an incomplete non-flat manifold.

We introduce the notations $g := g_N + e^{2 \varphi} \Tilde g_K$ and $g' := e^{-2 \varphi} g_N + \Tilde g_K$, so that $g = e^{2 \varphi} g'$, and let $\nabla'$ be the Levi-Civita covariant derivative of $g'$. We recall that the restriction of $\Tilde D_K$ to $N \times \Tilde K$ is the Levi-Civita of the metric $g$. Let $k \in \Tilde K$, and $X, Y \in T (N \times \{ k \})$. Now, we use the formula for the Levi-Civita connection under conformal change \cite[Theorem 1.159, a)]{Besse} and we obtain:
\begin{align*}
\nabla'_X Y = (\Tilde D_K)_X Y - d \varphi (X) Y - d \varphi (Y) X + g (X,Y) \Tilde D_K \varphi.
\end{align*}
We identify $N \times \{ k \}$ with $N$ in the canonical way, and using again the formula of conformal change for the metric $g'\vert_N = e^{-2 \varphi} g_N$, one obtains:
\begin{align*}
\nabla'_X Y = \Tilde D_X Y - d \varphi (X) Y - d \varphi (Y) X + g_N (X,Y) \Tilde D \varphi.
\end{align*}
Combining these two equations and remarking that $g(X, Y) \Tilde D_K \varphi = g_N (X,Y) \Tilde D \varphi $ we obtain $(\Tilde D_K)_X Y = \Tilde D_X Y$, which means that $N \times \{k \}$ is totally geodesic in $N \times \Tilde K$.

Suppose now that $q' \neq 0$. Let $X \in T \RR^{q'}$ be a parallel vector field of norm $1$. It induces canonically a parallel vector field of norm $1$, still denoted by $X$, on the Riemannian manifold $(N \times \Tilde K, g_N + e^{2 \varphi} \Tilde g_K)$. We claim that $X$ is tangent to $\Tilde K$. Indeed, for any $k \in \Tilde K$, the projection of $X$ onto $T (N \times \{ k \})$ is parallel because $N \times \{ k \}$ is totally geodesic. However, $(N, g_N)$ is irreducible and of dimension greater than $2$, so it does not admit a non-zero parallel vector field, thus this projection is equal to zero. Now we remark that $g'$ is a product metric, so $\nabla'_X X \in T \Tilde K$ and another use of the formula for the Levi-Civita connection under conformal change gives:
\begin{align*}
0 = (\Tilde D_K)_X X = \nabla'_X X + 2 d \varphi (X) X - g'(X,X) \nabla \varphi = \nabla'_X X - g'(X,X) \nabla \varphi
\end{align*}
because $\varphi$ is a function of $N$. Thus $T \Tilde K \ni \nabla'_X X = g'(X,X) \nabla \varphi$, and $g'(X,X) \neq 0$ so $\nabla \varphi \in T \Tilde K$ and $\nabla \varphi \in T N$, again because $\varphi$ is a function of $N$, which implies $\nabla \varphi = 0$ and $\varphi$ is constant. This is absurd because of the $\pi_1 (M)$-equivariance of $\varphi$, so $q' = 0$ and we conclude that $(N \times \Tilde K, g_N + e^{2 \varphi} \Tilde g_K)$ is irreducible, thus it is the non-flat part of the LCP manifold.
\end{proof}

In particular, the dimension of the non-flat part of the universal cover of an LCP manifold can be of any integer higher or equal to $2$.

These observations lead to the definition of reducible LCP manifolds:

\begin{definition}
A LCP manifold is called reducible if it arises from the previous construction, up to a finite covering. A non-reducible LCP manifold is called irreducible.
\end{definition}

\subsection{Similarity ratios of $\pi_1(M)$}

In the known examples of LCP manifolds, the similarity ratios are always algebraic numbers because they are roots of characteristic polynomials of matrices with coefficients in $\ZZ$. We will prove that this property is always true.

\begin{prop} \label{similratio}
Let $(M,c,D)$ be an LCP manifold. For any $\gamma \in \pi_1(M)$, the ratio of $\gamma$ viewed as a similarity of $(\Tilde M, h_D)$ is a unit of an algebraic number field.
\end{prop}
\begin{proof}
Let $\gamma \in \pi_1(M)$ and let $\lambda$ be its similarity ratio. For any $a \in \RR^q$ we will denote by $\tau_a$ the translation by $a$ in $\RR^q$, so $\RR^q$ is naturally identified with the space of translations. The restriction of $\gamma$ to $\RR^q$ can be written as $\gamma_E =: \tau_\alpha \circ \lambda \iota$ where $\iota$ is an isometry of $\RR^q$ endowed with the metric induced by $h_D$, and $\alpha \in \RR^q$.

Since $\Bar{P}^0$ is an abelian Lie group, the group $\RR^q \times \Bar{P}^0$ is abelian too. We define the group automorphism $\phi : \RR^q \times \Bar{P}^0 \to \RR^q \times \Bar{P}^0$ by
\begin{equation}
\phi (\tau_a, p) := \gamma (\tau_a, p) \gamma^{-1} = (\tau_{\lambda \iota a}, \gamma_N p \gamma_N^{-1}).
\end{equation}
Our proof relies on the crucial fact that the group $\Gamma_0 := \pi_1 (M) \cap (\mathrm{Sim} (\RR^q) \times \Bar{P}^0)$ defined in Section~\ref{foliation} is a full lattice in $\RR^q \times \Bar{P}^0$ by \cite[Lemma 4.18]{Kou}.

The preimage of $\Gamma_0$ by the Lie group exponential map is a full lattice $\Gamma_0'$ of the Lie algebra of $\RR^q \times \Bar{P}^0$, which is canonically identified with $\RR^{q+t}$, for some $t \ge 1$. The differential of $\phi$ at $e$ is a linear map satisfying $d_e \phi (\Gamma_0') \subset \Gamma_0'$ because $\phi (\Gamma_0) \subset \Gamma_0$. Moreover, $\phi$ is invertible and the symmetry between $\gamma$ and $\gamma^{-1}$ in the previous discussion gives that $d_e \phi^{-1} (\Gamma_0') \subset \Gamma_0'$. Thus, if we take a basis $\mathcal B$ of the lattice $\Gamma_0'$, the matrix $A := \mathrm{Mat}_\mathcal B (d_e \phi)$ is in  $\mathrm{GL}_{q+t} (\ZZ)$. But $\phi$ stabilizes $\RR^q$ and its restriction to this space coincides with $\lambda \iota$. It means that there exists a complex number $z$ of modulus $1$ such that $\lambda z$ and $\lambda \bar z$ are roots of the characteristic polynomial $\chi_A$ of $A$. Since $A \in \mathrm{GL}_{q+t} (\ZZ)$, $\lambda z$ and $\lambda \bar z$ are units of the algebraic field $K$ generated by the roots of $\chi_A$. Thus, $\lambda^2 = (\lambda z) (\lambda \bar z)$ is a unit of $K$, and therefore $\lambda$ is a unit in a quadratic extension of $K$.
\end{proof}

\section{Examples of LCP manifolds} \label{SectionExLCP}

We begin this section by stating a well-known result which will be useful for constructing LCP manifolds:

\begin{prop} \label{quotientsem}
Let $G$ be a discrete topological group acting on a manifold $M$. Let $D \trianglelefteq G$ be a normal subgroup. Then, $G/D$ acts on $M/D$, and $(M/D)/ (G/D)$ is in bijection with $M/G$.

If moreover $D$ and $G/D$ act freely and properly discontinuously on $M$ and $M/D$ respectively, then so does $G$ on $M$. In particular, $(M/D)/(G/D)$ and $M/G$ are diffeomorphic manifolds.
\end{prop}
\begin{proof}
The action of $G/D$ on $M / D$ is given by $g D \cdot D x := D g x$ for any $(g, x) \in G \times M$.

We define $\phi : (M/D)/(G/D) \to M/G$ by $\phi (G/D \cdot D x) := G x$ for any $x \in M$. This map is clearly surjective. In addition, if there are $(x,y) \in M^2$ such that $G x = G y$, there exists $g \in G$ such that $g x = y$, implying $g D \cdot D x = D y$ and then $G/D \cdot D x = G/D \cdot D y$, so $\phi$ is one-to-one.

Now, assume that $D$ and $G/D$ act freely and properly discontinuously on $M$ and $M/D$ respectively. Let $g \in G$ and $x \in M$ such that $g x = x$. Then, $g D \cdot D x = D x$, so $g D = 1_{G/D}$ because $G/D$ acts freely on $M/D$, implying $g \in D$, and $g = 1_G$ because $D$ acts freely on $M$. Thus $G$ acts freely on $M$.

To see that $G$ acts properly discontinuously on $M$, we pick a compact $K \subset M$. Let $g \in G$ satisfying $(g K) \cap K \neq \emptyset$. Since $D K$ is a compact subset of the manifold $M/D$, the set
$\{ g' D \in G/D \ \vert \ g' D \cdot (D K) \cap (D K) \neq \emptyset \}$ is finite: let $(g_j D)_{ j \in J}$ be the family of its elements, where $J$ is a finite set. Now, since $(g K) \cap K \ne \emptyset$, we also have $g D \cdot D K \cap D K \neq \emptyset$, so there is $j \in J$ such that $g_j D = g D$. This show that we can find $d \in D$ with $d g_j = g$ because $D$ is normal. Then, $(d g_j K) \cap K \neq \emptyset$. But there are only finitely many elements $d \in D$ satisfying this property because $D$ acts properly discontinuously on $M$. Let $(d_{j,i})_{i \in I_j}$ be the family of these elements, where $I_j$ is a finite set for every $j \in J$. Consequently, there exist $j \in J$ and $i \in I_j$ such that $g = d_{j,i} g_j$, and conversely any element of this form satisfy $(g K) \cap K \neq \emptyset$. Thus,
\[
\vert \{ g \in G \ \vert \ (g K) \cap K \neq \emptyset \} \vert = \sum\limits_{j \in J} \vert I_j \vert < + \infty,
\]
so $G$ acts properly discontinuously on $M$.

Finally, denote by $\pi_D : M \to M / D$, $\pi_{G/D} : M/D \to (M/D)/(G/D)$ and by $\pi_G : M \to M/G$ the canonical projections. One has the following commutative diagram:
\begin{equation}
\begin{tikzcd}
M \arrow[rdd, "\phi \circ \pi_{G/D} \circ \pi_D = \pi_G"] \arrow[dd, "\pi_{G/D} \circ \pi_D"'] \\
\\
(M/D)/(G/D) \arrow[r, "\phi"] & M/G
\end{tikzcd}
\end{equation}
and \cite[Theorem 4.29]{Lee} implies that $\phi$ is smooth.
\end{proof}

\subsection{General construction} \label{genconstruction}
Inspired by the known examples, we will now make a more general construction which includes all the models of LCP manifolds previously described.

Let $\Bar{\mathcal N}$ be a compact manifold. We will denote by $\mathcal N$ its universal cover and by $\Gamma$ its fundamental group, so $\Bar{\mathcal N} \simeq \mathcal N / \Gamma$. Let $p \in \NN$ and let $\phi : \Gamma \to \mathrm{Aff}_p (\ZZ)$ be a group morphism, where $\mathrm{Aff}_p (\ZZ) := \RR^p \rtimes \mathrm{GL}_p (\ZZ)$ is the set of affine transformations of $\RR^p$ with linear part in $\mathrm{GL}_p (\ZZ)$. We denote by $\phi_L : \Gamma \to \mathrm{GL}_p (\ZZ)$ the group morphism associating to $\gamma \in \Gamma$ the linear part of $\phi (\gamma)$.

We consider the simply connected manifold $\Tilde M := \RR^p \times \mathcal N$. Let $D \simeq \ZZ^p$ be the group of translations $\Tilde M \ni (a,x) \mapsto (a + z, x)$ for $z \in \ZZ^p$. Let $H$ be the group defined by $H := (\phi, \mathrm{id}) (\Gamma) = \lbrace (\phi (\gamma), \gamma) \vert \gamma \in \Gamma \rbrace \subset \mathrm{Aff}_p(\ZZ) \times \Gamma \subset \mathrm{Diff} (\Tilde M)$. Let $G$ be the subgroup of $\mathrm{Diff} (\Tilde M)$ generated by $D$ and $H$. It is clear that $D$ is a normal subgroup of $G$ and $G := D \rtimes H$. We claim that $G$ acts freely, properly discontinuously and cocompactly on $\Tilde M$. Indeed, one has $\Tilde M / D \simeq (S^1)^p \times \mathcal N$ and $H$ acts freely on this quotient because $\Gamma$ acts freely on $\mathcal N$. Moreover, $H$ also acts properly discontinuously because the map $(S^1)^p \times \mathcal N \to \mathcal N$ being proper and $H$ acting separately on $\RR^p$ and $\mathcal N$, it is sufficient to observe that $\Gamma$ acts properly on $\mathcal N$. In addition, this action is cocompact because $\Gamma$ acts cocompactly on $\mathcal N$ and $(S^1)^p$ is compact. Altogether, by Proposition~\ref{quotientsem} the quotient $\Tilde M / G$ is a compact manifold which we denote by $Q(\Bar{\mathcal N}, \phi)$, and whose fundamental group is $G$.

We now wish to construct an LCP structure on $Q (\Bar{\mathcal N}, \phi)$. To do so, we assume that the following conditions hold:

\begin{enumerate}[label=(\subscript{J}{{\arabic*}})]
\item there exist $\delta \in \NN$, a decomposition $\RR^p =: E_1 \oplus \ldots \oplus E_\delta$ stabilized by the action of $\phi_L (\Gamma)$, and a positive definite bilinear form $b$ on $\RR^p$ such that the previous decomposition is orthogonal with respect to $b$ and for any $1 \le k \le \delta$, the restriction of $\phi_L (\Gamma)$ to $(E_k, b \vert_{E_k})$ consists of similarities;
\item $O (E_1, b \vert_{E_1})$ does not contain $\phi_L (\Gamma) \vert_{E_1}$.
\end{enumerate}

\begin{rem} \label{dimension}
In particular, condition $(J_1)$ allows us to define a group morphism $\Lambda : \Gamma \to (\R_+^*)^\delta$ which associates to any $\gamma \in \Gamma$ the $\delta$-tuple given by the similarity ratios of $\phi_L(\gamma) \vert_{E_1}, \ldots, \phi_L (\gamma) \vert_{E_\delta}$. For any $1 \le k \le \delta$, $\Lambda_k$ will denote the $k$-th coordinate of $\Lambda$. Condition $(J_2)$ implies that $2 \le p$. Indeed, if $p = 1$, $\phi_L (\Gamma) \subset \{\pm 1\} = O (\RR^p, b)$.
In addition, from $(J_2)$ we also deduce that $2 \le \delta$, because otherwise $\RR^p = E_1$ and there would exist an element $\gamma \in \Gamma$ such that $\pm 1 \neq \Lambda_1 (\gamma)^p = \det \phi_L (\gamma) = \pm 1$, which is absurd. In particular, $Q (\Bar{\mathcal N}, \phi)$ has dimension at least $3$.
\end{rem}

We will need the following standard lemma:

\begin{lemma} \label{equifun}
Let $\mathcal Z$ be a smooth manifold on which a group $\Gamma'$ acts freely and properly discontinuously, so in particular $\mathcal Z / \Gamma'$ is a smooth manifold. Let $\rho : \Gamma' \to \RR_+^*$ be a group morphism. Then, there exists a function $f \in C^\infty (\mathcal Z, \RR)$ such that for any $\gamma \in \Gamma'$, $\gamma^* e^{2 f} = \rho (\gamma)^2 e^{2 f}$.
\end{lemma}
\begin{proof}
Let $\pi_\mathcal Z : \mathcal Z \to \mathcal Z / \Gamma'$ be the canonical submersion. We define the oriented line bundle $L := \mathcal Z \times_{\rho^{-1}} \RR$. Since any orientable line bundle is trivial, there exists $s : \mathcal Z / \Gamma' \to L$ a nowhere vanishing smooth section of $L$. Then, after replacing $s$ by $-s$ if necessary, there is a function $f : \mathcal Z \to \RR$ such that for all $x \in \mathcal Z$ one has $s (\pi_\mathcal Z (x)) = [x, e^{f (x)}]$. Moreover, for any $\gamma \in \Gamma'$, we have
\[
[x, e^{f (x)}] = s (\pi_\mathcal Z (x)) = s (\pi_\mathcal Z (\gamma x)) = [\gamma x, e^{f (\gamma x)}] = [x, \rho (\gamma)^{-1} e^{f (\gamma x)}],
\]
which implies $\rho (\gamma) e^{f (x)} = e^{f (\gamma x)}$, so the function $f$ has the desired equivariance property.
\end{proof}

We are now in position to construct an LCP structure on $Q (\Bar{\mathcal N}, \phi)$.

\begin{prop} \label{geneconst}
Under the assumptions $(J_1)$, $(J_2)$ there exists an LCP structure on $Q (\Bar{\mathcal N}, \phi)$. The LCP manifold obtained in this way has rank equal to $\mathrm{rk} (\Lambda_1 (\Gamma))$ and the flat part of its universal cover contains $E_1$.
\end{prop}
\begin{proof}
Let $\bar g$ be any Riemannian metric on $\Bar{\mathcal N}$ and let $\Tilde g$ be its lift to $\mathcal N$. Let $f \in C^\infty (\mathcal N, \RR)$ be the function given by Lemma~\ref{equifun} applied to the morphism $\rho := \Lambda_1$. By definition, an element $\gamma \in \Gamma$ acts as a similarity of ratio $\Lambda_1 (\gamma)$ on $(\mathcal N, g := e^{2f} \Tilde g)$.

For any $2 \le k \le \delta$, we define the morphism 
\begin{align}
\rho_k : \Gamma \to \RR_+^*, && \gamma \mapsto \Lambda_1 (\gamma) / \Lambda_{k} (\gamma).
\end{align}
By Lemma~\ref{equifun}, we know that the set
\begin{equation}
\mathcal F_{eq} (k) := \lbrace f \in C^\infty (\mathcal N, \RR) \ \vert \ \forall \gamma \in \Gamma, \gamma^* e^{2 f} = \rho_k (\gamma)^2 e^{2 f} \rbrace.
\end{equation}
is non-empty.

We identify the tangent bundle $T \RR^p$
with $\RR^p \times \RR^p$ in the canonical way, and the bilinear form $b$ thus defines a Riemannian metric on $\RR^p$. Then, we define a metric $h$ on $\Tilde M = \RR^p \times \mathcal N$ by
\begin{equation}
h := b\vert_{E_1} + \sum\limits_{k = 2}^\delta e^{2 f_k} b\vert_{E_k} + g,
\end{equation}
where for all $2 \le k \le \delta$, $f_k \in \mathcal F_{eq} (k)$.

One clearly has for any $T \in D$ that $T^* h = h$. For any $\gamma \in \Gamma$, one has
\begin{align*}
(\phi (\gamma),\gamma)^* h =& \Lambda_1 (\gamma)^2 b\vert_{E_1} + \sum\limits_{k = 2}^\delta \gamma^* e^{2 f_k} \Lambda_k (\gamma)^2 b\vert_{E_k} + \gamma^* g \\
=& \Lambda_1 (\gamma)^2 b\vert_{E_1} + \sum\limits_{k = 2}^\delta \left(\frac{\Lambda_1 (\gamma)}{\Lambda_k (\gamma)}\right)^2 e^{2 f_k} \Lambda_k (\gamma)^2 b\vert_{E_k} + \Lambda_1 (\gamma)^2 g \\
=& \Lambda_1 (\gamma)^2 h.
\end{align*}

Since $G = D \rtimes H$, the elements of $G$ act as similarities, and $\Tilde g$ is a similarity structure on $Q (\Bar{\mathcal N}, \phi)$ which is not Riemannian because of condition $(J_2)$.

It remains to prove that $(\Tilde M, h)$ is non-flat with reducible holonomy. But $E_1$ is a Riemannian factor of $\Tilde M$, so the claim follows from Remark~\ref{CauchyborderLCP}, because the Cauchy border contains the set $E_1 \times \partial \mathcal N$ which is infinite.
\end{proof}

\begin{example}
We consider the matrix
\begin{align}
B := \left(
\begin{matrix}
1 & 1 \\
1 & 2
\end{matrix}
\right) \in \mathrm{SL}_2 (\ZZ).
\end{align}
Let $q \ge 1$. Let $A \in \mathrm{SL}_{2 q} (\ZZ)$ which is the matrix diagonal by blocks with $q$ times the block $B$. We consider a bilinear symmetric form $b_0$ on $\RR^2$ for which the two eigenspaces of $B$ are orthogonal, and we define the symmetric bilinear form $b := \bigoplus\limits_{k=1}^q b_0$ on $\RR^{2q}$. We consider $\Bar {\mathcal N} := S^1$, whose fundamental group is $\Gamma := \ZZ$, and the group morphism $\phi : \Gamma \to \mathrm{SL}_{2 q} (\ZZ)$, $n \mapsto A^n$. By Proposition~\ref{geneconst}, $Q (S^1, \phi)$ admits an LCP structure whose universal cover has a flat part of dimension $q$. Thus the dimension of the flat part can be any integer.
\end{example}

As an application of Proposition~\ref{geneconst}, we will show that on any OT-manifold (recall that they were defined in Example~\ref{OTex}) carries an LCP structure. The proof of this fact just relies on the remark that an OT-manifold is a particular case of the construction above.

\begin{cor} \label{LCPonOT}
Any OT-manifold $X(K, U)$ can be endowed with an LCP structure.
\end{cor}
\begin{proof}
We use the notations of Example~\ref{OTex}. By definition one has
\[
X(K, U) = (H^s \times \CC^t) / (\mathcal O_K \rtimes U)
\]
so its universal cover is naturally isomorphic to $(\RR_+^*)^s \times \RR^s \times \RR^{2 t} \simeq \RR^s \times \RR^{s + 2 t}$ using the logarithm map. By construction, the group $\Gamma := p_{\RR^s} \circ \ell (U)$ acts freely, properly discontinuously and cocompactly on $\mathcal N := \RR^s$ because it is a full lattice. Moreover, $U$ is of rank $s$, so $\psi := (p_{\RR^s} \circ \ell)^{-1}$ is a group isomorphism between $\Gamma$ and $U$.

Let $\mathcal B = (e_1, \ldots, e_{s+2t})$ be the canonical basis of $\RR^{s + 2t}$. Let $\mathcal B'$ be a basis of the lattice $\sigma(\mathcal O_K)$, so in particular another basis of $\RR^{s + 2 t}$. With respect to the basis $\mathcal B'$, the action of $U$ restricted to $\RR^{s + 2 t}$ consists of multiplication by matrices of $\mathrm{GL}_{s+ 2t} (\ZZ)$ because $U$ preserves $\sigma(\mathcal O_K)$. This induces a group morphism $U \to \mathrm{GL}_{s+ 2t} (\ZZ)$ and then a group morphism $\phi : \Gamma \to \mathrm{GL}_{s+ 2t} (\ZZ)$ using the isomorphism $\psi$ between $\Gamma$ and $U$. Consequently, $X (K, U) \simeq Q (\mathcal N / \Gamma, \phi)$.

It is now sufficient to check that conditions $(J_1)$, $(J_2)$ hold, so we can apply Proposition~\ref{geneconst} to conclude. Let $b$ be the Euclidean metric on $\RR^{s + 2t}$ for which $\mathcal B$ is orthonormal. By construction, for any $\gamma \in \Gamma$, the matrix of $\phi (\gamma)$ in the basis $\mathcal B$ is of the form
\begin{equation} \label{matricebaseB}
\left(
\begin{matrix}
\sigma_1 (u) & & & & & \\
 & \ddots & & & & \\
 & & \sigma_s (u) & & &\\
& & & \vert \sigma_{s+1} (u) \vert O_1(u) & & \\
& & & & \ddots & \\
& & & & & \vert \sigma_{s+t} (u) \vert O_t(u)
\end{matrix}
\right)
\end{equation}
where $u \in \mathcal O^{\times,+}_K$ and $O_1(u), \ldots, O_t(u) \in \mathrm{SO}_2 (\RR)$ are the rotations induced by the multiplications by the complex numbers $\sigma_{s+1} (u) / \vert \sigma_{s+1} (u)\vert$ on $\CC$. Then, the spaces
\[
E_j := \Span (e_j)
\]
for $1 \le j \le s$ and
\[
E_{s + j} := \Span (e_{s+ 2 j - 1}, e_{s+2 j})
\]
for $1 \le j \le t$ give a decomposition of $\RR^{s + 2t}$ in orthogonal subspaces stable by the action of $\phi (\Gamma)$, so $(J_1)$ is verified because of the form of the matrix \eqref{matricebaseB}. Finally, $\sigma_1$ is injective so for any $u \in U$, $\sigma_1 (u) = 1$ implies $u = 1$. Thus there exists $u \in U$ such that $\sigma_1 (u) \in (0,1)$ (because we recall that $\sigma_1(u) > 0$ by construction) so $(J_2)$ holds.
\end{proof}

It is important to notice that the LCP metrics constructed by using the proof of Proposition~\ref{geneconst} on OT-manifolds with the approach of Corollary~\ref{LCPonOT} do not contain the LCK structures introduced in \cite{OT} when $t = 1$. However, we can extend the family of LCP metrics defined in the proof of Proposition~\ref{geneconst}. For any $2 \le k,k' \le \delta$ with $k \neq k'$, consider the morphism
\begin{align}
\rho_{k,k'} : \Gamma \to \RR_+^*, && \gamma \mapsto \Lambda_1 (\gamma) / \sqrt{\Lambda_{k} \Lambda_{k'}} (\gamma),
\end{align}
and let $b_{k,k'} : E_k \times E_{k'} \to \RR$ be a bilinear form satisfying $\phi (\gamma)^* b_{k,k'} = \sqrt{\Lambda_{k} \Lambda_{k'}} (\gamma) b_{k,k'}$ for any $\gamma \in \Gamma$ (such forms always exist, since we can take $b_{k,k'} = 0$), and let $f_{k,k'}$ be an element of the set
\begin{equation}
\mathcal F_{eq} (k,k') := \lbrace f \in C^\infty (\mathcal N, \RR) \ \vert \ \forall \gamma \in \Gamma, \gamma^* e^{2 f} = \rho_{k,k'} (\gamma)^2 e^{2 f} \rbrace.
\end{equation}
Then, we consider the metric $h$ on $\RR^p \times \mathcal N$ defined by
\begin{equation}
h := b\vert_{E_1} + \sum\limits_{k = 2}^\delta e^{2 f_k} b\vert_{E_k} + \sum\limits_{k = 2}^\delta \sum\limits_{k' = 2}^\delta e^{2 f_{k,k'}} b_{k,k'} + g.
\end{equation}
If the functions $f_{k,k'}$ are taken small enough on a relatively compact fundamental domain of $\mathcal N$, $h$ is positive definite, and an argument similar to the one used in Proposition~\ref{geneconst} shows that the elements of the group $G$ act as $h$-similarities.

On an OT-manifold with $t = 1$, the LCK metric on its universal cover $H^s \times \CC$ defined in \cite{OT} is of the form
\begin{equation}
h := \left( \prod\limits_{j = 1}^s \frac{1}{2 y_j} \right) \left( \sum\limits_{k, k' = 1}^s \frac{1 + \delta_k^{k'}}{4 y_k y_{k'}} (d x_k \otimes d x_{k'} + d y_k \otimes d y_{k'}) \right) + d x_{s+1}^2 + d y_{s+1}^2,
\end{equation}
where $z_k := x_k + i y_k$, $1 \le k \le s+1$ are the canonical complex coordinates and $\delta_k^{k'}$ is the Kronecker symbol. This falls on the construction above by taking $E_1 := \Span(e_{s+1}, e_{s+2})$ (with the basis introduced in Corollary~\ref{LCPonOT}), together with the functions
\[
f_{k,k'} :=  \left( \prod\limits_{j = 1}^s \frac{1}{2 y_j} \right) \frac{1}{y_k y_{k'}}
\]
and the bilinear forms $b_{k,k'} := d x_k \otimes d x_{k'}$.

\subsection{Rank of an LCP manifold} Our next goal is to construct LCP manifolds of arbitrary rank using again Proposition~\ref{geneconst} again. For this purpose, we need a special family of commuting matrices, which will be constructed by means of number theory. This makes the object of the two following two lemmas:

\begin{lemma} \label{exfield}
For any $n \in \NN$ there exists a cyclic, totally real and monogenic algebraic number field of degree $p \ge n + 1$.
\end{lemma}
\begin{proof}
Let $n \in \NN$, and let $m \ge 2 n + 3$ be a prime number. Let $K$ be the maximal real subfield of the $m$-th cyclotomic extension. Then $K$ is an extension of $\QQ$ of degree $p := (m-1) / 2 \ge n + 1$, which is totally real, monogenic by \cite[Proposition 2.16]{Wash}, and cyclic.
\end{proof}

\begin{lemma} \label{dmatrix}
Let $n \ge 2$. There exists an integer $p \ge n + 1$ and diagonalizable matrices $A_1, \ldots, A_n \in \mathrm{GL}_p (\ZZ)$ with the following properties:
\begin{itemize}
\item The matrices $A_1, \ldots, A_n$ commute, so their are codiagonalizable.
\item Let $(e_1, \ldots, e_p)$ be a common basis of diagonalization for $A_1, \ldots, A_n$. For any $1 \le k \le p$, let $E_k = \Span (e_k)$, and denote by $\lambda_k (A_l)$ the eigenvalue of $A_l$ associated to the eigenspace $E_k$. Then, the subgroup $\langle \vert \lambda_1 (A_1) \vert, \ldots, \vert \lambda_1 (A_n) \vert \rangle$ of $\RR_+^*$ has rank $n$.
\end{itemize}
\end{lemma}
\begin{proof}
Let $K$ be a cyclic, totally real and monogenic algebraic number field of degree $p \ge n + 1$, which exists by Lemma~\ref{exfield}. There is an algebraic integer $\alpha$ such that $\alpha$ generates a power basis of $K$, in particular $K = \QQ [\alpha]$. By Dirichlet's units theorem, the group of units of $\QQ[\alpha]$ has rank $p - 1$. Since $p-1 \ge n$, we can take $n$ independent fundamental units $u_1, \ldots, u_n$ in $\QQ [\alpha]$. By monogeneity, there are polynomials $P_1, \ldots, P_n \in \ZZ_{p-1} [X]$ such that $P_l (\alpha) = u_l$ for any $1 \le l \le n$.

Now, let $A \in \mathrm{GL}_p(\ZZ)$ be the companion matrix of the minimal polynomial of $\alpha$ and let $A_l := P_l (A) \in \mathrm{GL}_p(\ZZ)$ for $1 \le l \le n$. Since the minimal polynomial of $\alpha$ is irreducible over $\QQ$, it is separable and $A$ is diagonalizable in $\RR$, with eigenvalues equal to the conjugates of $\alpha$, namely $\alpha, \sigma (\alpha), \ldots, \sigma^{p-1}(\alpha)$, where $\sigma$ is a generator of the (cyclic) Galois group of $\QQ [\alpha]$. Then, the matrices $A_l$ are diagonalizable with eigenvalues $u_l, \sigma(u_l), \ldots, \sigma^{p-1}(u_l)$. Moreover, their determinants are $\Pi_{k=0}^{p-1} \sigma^k (u_l) = \pm 1$ because $u_l$ is a unit.

Finally, let $e_1$ be an eigenvector of $A$ for the eigenvalue $\alpha$. Then, $E_1 := \Span (e_1)$ is a one-dimensional eigenspace of any $A_l$ for the eigenvalue $u_l$, and $\langle u_1, \ldots, u_n \rangle$ is of rank $n$. We can complete $(e_1)$ in a basis of diagonalization of $A$ to obtain the last property of the lemma.
\end{proof}

The matrices defined in Lemma~\ref{dmatrix} will be used to define the morphism $\phi$ needed for the construction of Proposition~\ref{geneconst}, so we prove the following:

\begin{prop} \label{rankLCP}
Let $n \ge 1$. Let $p \ge n +1$ and $A_1, \ldots, A_n \in \mathrm{GL}_p (\ZZ)$ be the matrices given by Lemma~\ref{dmatrix}. The group $H := \langle A_1, \ldots, A_n \rangle$ is canonically isomorphic to $\ZZ^n$, defining a group isomorphism $\phi : \ZZ^n \to H$. Then, there exists a LCP structure on $Q ((S^1)^n, \phi)$ of rank $n$.

In particular, the rank of an LCP manifold can be any positive integer.
\end{prop}
\begin{proof}
We keep the notations of Lemma~\ref{dmatrix} in this proof. Let $\mathcal B$ be a basis adapted to the decomposition $E_1 \oplus \ldots \oplus E_p$ and let $b$ be the symmetric, positive definite bilinear form for which $\mathcal B$ is orthonormal. Then, the conditions $(J_1)$ and $(J_2)$ are satisfied, so by Proposition~\ref{geneconst} $Q ((S^1)^n, \phi)$ carries an LCP structure of rank $n$.
\end{proof}

\begin{example}
We can make an explicit computation of the matrices given by Lemma~\ref{dmatrix} in the case $n = 2$ by following the constructive approach of the proof. Taking $m = 7$ in the proof of Lemma~\ref{exfield} shows that $K := \QQ [2 \cos(\frac{2 \pi}{7})]$ is a totally real, monogenic, cyclic extension of $\QQ$ of degree $p = 3$. From now on, we denote by $\alpha := 2 \cos(\frac{2 \pi}{7})$. From \cite[Proposition 2.16]{Wash}, one has $\mathcal O_K = \ZZ [\alpha]$. The minimal polynomial of $\alpha$ is $X^3 + X^2 - 2 X - 1$, and its conjugates are $2 \cos(\frac{4 \pi}{7})$ and $2 \cos(\frac{6 \pi}{7})$. Let $\sigma$ be the automorphism of $K$ such that $\sigma (\alpha) = 2 \cos(\frac{4 \pi}{7})$. Then $\sigma^2 (\alpha) = 2 \cos(\frac{6 \pi}{7})$ and $\sigma^3 = \mathrm{id}_K$.

We claim that the (multiplicative) group $\langle \alpha, \sigma(\alpha) \rangle$ has rank $2$. Indeed, if there were $a,b \in \ZZ$ such that $\alpha^a = \sigma (\alpha)^b$,then the two vectors of $\RR^3$ given by
\begin{align*}
X_1 := (\ln \vert \alpha \vert, \ln \vert \sigma (\alpha) \vert, \ln \vert \sigma^2 (\alpha) \vert), &&
X_2 := (\ln \vert \sigma (\alpha) \vert, \ln \vert \sigma^2 (\alpha) \vert, \ln \vert \alpha \vert)
\end{align*}
would be collinear. But $X_1$ and $X_2$ have the same norm for the standard Euclidean metric in $\RR^3$ because the coefficients of $X_2$ are a permutation of the ones of $X_1$, so they are collinear if and only if $X_1 = \pm X_2$. But this is false because $\cos(\frac{2 \pi}{7}) \neq \cos(\frac{4 \pi}{7})^{\pm 1}$.

Now, we have the equality $\sigma (\alpha) = \alpha^2 - 2$. Thus, we consider the companion matrix of the minimal polynomial of $\alpha$:
\begin{equation}
A_1 := \left(
\begin{matrix}
0 & 0 & 1 \\
1 & 0 & 2 \\
0 & 1 & -1
\end{matrix}
\right),
\end{equation}
and the matrix
\begin{equation}
A_2 := A_1^2 - 2 \mathrm{I}_3 = \left(
\begin{matrix}
- 2 & 1 & -1 \\
0 & 0 & -1 \\
1 & -1 & 1
\end{matrix}
\right).
\end{equation}
One easily checks that eigenvectors corresponding to the eigenvalues $\alpha, \sigma(\alpha), \sigma^2 (\alpha)$ of $A_1$ can be taken respectively as
\begin{equation}
x_1 = \left(
\begin{matrix}
1 \\
\alpha + \alpha^2 \\
\alpha
\end{matrix}
\right), \quad
x_2 = \left(
\begin{matrix}
1 \\
\sigma(\alpha) + \sigma(\alpha)^2 \\
\sigma(\alpha)
\end{matrix}
\right), \quad
x_3 = \left(
\begin{matrix}
1 \\
\sigma^2(\alpha) + \sigma^2(\alpha)^2 \\
\sigma^2(\alpha)
\end{matrix}
\right),
\end{equation}
and they are eigenvectors of $A_2$ for the eigenvalues $\sigma (\alpha), \sigma^2 (\alpha), \alpha$ respectively.

Using these matrices, we can now give the explicit construction of an LCP manifold of rank $2$ following Proposition~\ref{rankLCP} and Proposition~\ref{geneconst}. On the manifold $\Tilde M := \RR^3 \times (\RR_+^*)^2$, the group
\begin{equation}
G := D \rtimes \langle (A_1, (\vert \alpha \vert, 1)), (A_2, (1, \vert \sigma (\alpha) \vert)) \rangle
\end{equation}
acts freely, properly discontinuously and cocompactly (here the group $D$ is defined as in Section~\ref{genconstruction}, as the group of translations $\ZZ^3$ acting on $\RR^3$). Let $(t_1, t_2)$ be the canonical coordinates of $(\RR_+^*)^2$. We define the metric
\begin{equation}
h_{x,t} := d x_1^2 + \phi_2(t_1,t_2)^2 d x_2^2 + \phi_3 (t_1, t_2)^2 d x_3^2 + t_2^2 d t_1^2 + t_1^2 d t_2^2.
\end{equation}
where
\begin{align}
\phi_2 (t_1, t_2) := \left\vert \frac{\alpha}{\sigma (\alpha)} \right\vert^{\ln (t_1) / \ln (\vert \alpha \vert)} \left\vert \frac{\sigma(\alpha)}{\sigma^2 (\alpha)} \right\vert^{\ln (t_2) / \ln (\vert \sigma(\alpha) \vert)} \\ \phi_3 (t_1, t_2) := \left\vert \frac{\alpha}{\sigma^2 (\alpha)} \right\vert^{\ln (t_1) / \ln (\vert \alpha \vert)} \left\vert \frac{\sigma(\alpha)}{\alpha} \right\vert^{\ln (t_2) / \ln ( \vert \sigma(\alpha) \vert)}.
\end{align}
The manifold $M := \Tilde M / G$ admits a non-Riemannian similarity structure given by $h$, which in turn defines an LCP structure of rank $2$ on $M$.
\end{example}

\section{Some open questions}

Some questions arise naturally from the analysis and the discussions done in the previous sections. We make here a non-exhaustive list of such ones, whose answers would lead to a better understanding of LCP manifolds. Throughout this section, we will use the notations of Section~\ref{foliation}.

First of all, it was noticed by Kourganoff \cite[Theorem 1.9]{Kou} that the dimension of the closures of the leaves, which are the elements of $\Bar{\mathcal F}$ in the setting of Theorem~\ref{Kou1.9}, may vary. However, in all the examples given in this article, this dimension is constant, so we ask the following:
\begin{itemize}
\item In the setting of Theorem~\ref{Kou1.9}, do all the elements of $\Bar{\mathcal F}$ have the same dimension?
\end{itemize}

We can propose a strategy to answer this first question. Indeed, assume that $\Bar{P}^0$ is simply connected, i.e. it is isomorphic to the group $\RR^t$ for some $t \in \NN$. Then, since the group $\Gamma_0$ is a full lattice in $\RR^q \times \Bar{P}^0 \simeq \RR^{q + t}$, the group $\Gamma_0$ is of rank $q + t$. In addition, for any $x \in N$ (the non-flat part), the closed leaf $\Bar {\mathcal F}_x = \pi (\RR^q \times {\Bar P}^0 x)$ has the same dimension as $\RR^q \times {\Bar P}^0 x$. As we already saw, this last manifold is isomorphic to the product of an Euclidean space with a flat torus so it is a Lie group, and $\Gamma_0$ acts freely and properly discontinuously on it. Consequently, $\Gamma_0 (\{ (0,x) \})$ is a lattice of $\RR^q \times {\Bar P}^0 x$ with rank equal to $q + t$. Thus
\begin{equation}
q + t = \mathrm{rank} (\Gamma_0) \le \dim (\RR^q \times {\Bar P}^0 x) \le q +t,
\end{equation}
and these inequalities turn out to be equalities, so $\Bar {\mathcal F}_x$ has dimension $q + t$. This leads to the following question, whose answer is positive in all the examples:
\begin{itemize}
\item Is the group $\Bar{P}^0$ simply connected, or equivalently is it isomorphic to $\RR^t$ for some $t \in \NN$?
\end{itemize}

In order to have a better understanding of the group $P$, we should specify how it acts on $N$. In \cite[Lemma 4.17]{Kou}, it was shown that $P$ acts freely on $N$, but the proof proposed seems incorrect, even if it does not modify the correctness of the rest of the article. The only result we can obtain is the one of Lemma~\ref{isomPpi}, stated previously. We thus ask:
\begin{itemize}
\item Does $P$ acts freely on $N$? If this is true, does $\Bar P$ acts freely on $N$?
\end{itemize}

In Section~\ref{genconstruction}, we have given a general construction to obtain LCP manifolds. Nevertheless, some points remain imprecise:
\begin{itemize}
\item What are the acceptable choices for the morphism $\phi$, given a compact manifold $\mathcal N$ (even without asking for conditions $(J_1)$ and $(J_2)$)?
\item Can we weaken conditions $(J_1)$ and $(J_2)$?
\end{itemize}

Finally, we remark that the only known LCK manifolds which are also LCP are the OT-manifolds for $t = 1$. A natural way to construct new examples would be to take extensions of OT-manifolds (see Definition~\ref{extensionDef}).
\begin{itemize}
\item Can an extension of an LCP manifold which is also LCK be an LCK manifold?
\item Are the OT-manifolds with $t = 1$ the only LCP manifolds which are also LCK?
\end{itemize}

\section*{Statements and Declarations}

The author has no relevant financial or non-financial interest to disclose.

Data sharing is not applicable to this article as no datasets were generated or analyzed during the current study. 

\renewcommand{\refname}

\end{document}